\documentclass[11pt]{amsart}
\usepackage[T1]{fontenc}
\usepackage{amsthm,amstext,amssymb}
\usepackage[american,english]{babel}

\allowdisplaybreaks

\makeatletter

\newcommand{\noun}[1]{\textsc{#1}}

\numberwithin{equation}{section}
\numberwithin{figure}{section}
\theoremstyle{plain}
\newtheorem{thm}{Theorem}[section]
  \theoremstyle{plain}
  \newtheorem{lem}[thm]{Lemma}
  \newtheorem{proposition}[thm]{Proposition}
  \theoremstyle{remark}
  \newtheorem{rem}[thm]{Remark}
 \theoremstyle{definition}
 \newtheorem*{defn*}{Definition}
  \theoremstyle{remark}
  \newtheorem*{rem*}{Remark} 
  \theoremstyle{plain}
  \newtheorem{cor}[thm]{Corollary}
 \theoremstyle{definition}
  \newtheorem{example}[thm]{Example}
  \theoremstyle{definition}
  \newtheorem{defn}[thm]{Definition}
  \theoremstyle{remark}
    \newtheorem{claim}[thm]{Claim}
  \theoremstyle{definition}
  \newtheorem{problem}[thm]{Problem}
\theoremstyle{definition}
  \newtheorem{lemma}[thm]{Lemma}
\theoremstyle{definition}
  \newtheorem{assumption}[thm]{Assumption}

\newcommand{\Rand}[1]{\marginpar{#1}}
     \renewcommand{\Rand}[1]{}
\newcommand{\be}[1]{\Rand{\vspace{0,6cm}\tt #1}\begin{equation}\label{#1}}
\newcommand{\bea}[1]{\Rand{\vspace{0,6cm}\tt #1}\begin{eqnarray}\label{#1}}
\newcommand{\beL}[2]{\Rand{\vspace{0,6cm}\tt #1}\begin{lemma}[#2]\label{#1}}
\newcommand{\beA}[2]{\Rand{\vspace{0,6cm}\tt #1}\begin{assumption}[#2]\label{#1}}
\newcommand{\beD}[2]{\Rand{\vspace{0,6cm}\tt #1}\begin{definition}[#2]\label{#1}}
\newcommand{\beT}[2]{\Rand{\vspace{0,6cm}\tt #1}\begin{theorem}[#2]\label{#1}}
\newcommand{\beP}[2]{\Rand{\vspace{0,6cm}\tt #1}\begin{proposition}[#2]\label{#1}}
\newcommand{\beC}[2]{\Rand{\vspace{0,6cm}\tt #1}\begin{conjecture}[#2]\label{#1}}
\newcommand{\beCor}[2]{\Rand{\vspace{0,6cm}\tt #1}\begin{corollary}[#2]\label{#1}}

\def\qed{{\hfill $\square$ \bigskip}}

\def\P{\mathbb P}
\def\E{\mathbb E}
\def\R{\mathbb R}

\def\LL{{  L}}

\makeatother

\begin{document}

\title [Large mass creation]{Superdiffusions with large mass creation --- construction and growth estimates}


\author[Z.-Q. Chen]{Zhen-Qing Chen}
\address[Z.-Q. Chen]{Department of Mathematics\\ University of Washington\\ Seattle, WA
98195, USA}
 \email{zqchen@uw.edu}
\urladdr{http://www.math.washington.edu/$\sim$zchen/}
 

\author[J. Engl\"ander]{J\'anos Engl\"ander}
\address[J. Engl\"ander]{Department of Mathematics\\ University of Colorado\\
 Boulder, CO-80309-0395}
 
\email {Janos.Englander@Colorado.edu.}
\urladdr{http://euclid.colorado.edu/$\sim$englandj/MyBoulderPage.html}
\thanks{ J. Engl\"ander's research was partially supported by NSA grant ~H982300810021}



\date{\today}

\begin{abstract}
 Superdiffusions corresponding to differential operators of the form $\LL u+\beta u-\alpha u^{2}$ with large mass creation term $\beta$ are studied. Our construction for superdiffusions with large mass creations works for the branching mechanism $\beta u-\alpha u^{1+\gamma},\ 0<\gamma<1,$ as well.

Let $D\subseteq\mathbb{R}^{d}$ be a domain in $\R^d$. 
When $\beta$ is large, the generalized principal eigenvalue $\lambda_c$ of $L+\beta$ in $D$ is typically infinite.  Let
$\{T_{t},t\ge0\}$  denote the Schr\"odinger semigroup of $L+\beta$ in $D$ with zero Dirichlet boundary condition. Under the mild assumption that there exists an $0<h\in C^{2}(D)$ so that $T_{t}h$ is finite-valued for all $t\ge 0$,  we show that there is a unique $\mathcal{M}_{loc}(D)$-valued Markov process that satisfies a log-Laplace equation in terms of the minimal nonnegative solution to 
 a semilinear initial value problem. 
Although for super-Brownian motion (SBM) this assumption requires 
 $\beta$ be 
less than quadratic, the quadratic case will be  treated as well.

When $\lambda_c = \infty$, the usual machinery, including martingale methods and PDE as well as other similar techniques cease to work effectively, both for the construction and for the investigation of the large time behavior of the superdiffusions.
In this paper, we develop the following two new techniques in the study of local/global growth  of  mass and for the spread of the superdiffusions:  
\begin{itemize}
\item a generalization of the Fleischmann-Swart `Poissonization-coupling,' linking superprocesses with branching diffusions;  
\item  the introduction of a new concept: the `{\it $p$-generalized principal eigenvalue.}'
\end{itemize} 
The precise growth rate for the total population of SBM with $\alpha(x)=\beta(x)=1+|x|^p$ for $p\in[0,2]$ 
is given   in this paper. 
\end{abstract}
\maketitle
\tableofcontents

\section{Introduction}  \label{S:intro}
\subsection{Superdiffusions}

Like Brownian motion, {\it super-Brownian motion} is also a building block in stochastic analysis. 
Just as  Brownian motion is a prototype of the more general diffusion processes, super-Brownian motion is a particular {\it superdiffusion}.
Superdiffusions are measure-valued Markov processes, but  unlike for branching diffusions, the values of superdiffusions taken  for $t>0$ are no longer discrete measures. Intuitively, such a process describes the evolution of a random cloud in space, or random mass distributed in space,   creating more mass at some regions while annihilating mass at some others along the way.

The usual way of defining or constructing a superdiffusion $X$ is:
\begin{enumerate}
\item as a measure valued Markov process  via its Laplace functional; or
\item as a scaling limit of branching diffusions.
\end{enumerate}
The second approach means that $X$  arises as the short lifetime and
high density diffusion limit of a \emph{branching particle system},
which can be described as follows: in the $n^{\mathrm{th}}$
approximation step each particle has mass $1/n$ and lives for a random
lifetime which is exponentially distributed with mean $1/n$. While a particle is
alive, its motion is described by a diffusion process  in $D$ with infinitesimal generator
 $L$ (where $D$ is a subdomain of $\R^d$  and the diffusion process is killed upon leaving $D$). 
At the end of its life, the particle located at $x\in D$ dies and is replaced
by a random number of offspring situated at the same location $x$. The law of the number of descendants is spatially varying
such that the number of descendants has mean $1+\frac{\beta(x)}{n}$ and variance   $2\alpha(x)$. 
Different particles experience branching and migration independently of each other; the branching of a given particle may interact with its motion, as the branching mechanism is spatially dependent.
 Hence  a superdiffusion can be  described by the quadruple $(L,  \beta,\alpha; D)$, where $L$ is the second order elliptic operator corresponding to the underlying spatial motion, 
$\beta$ (the `mass creation term') describes the growth rate of the superdiffusion\footnote{In a region where $\beta<0$, one actually has mass annihilation.},
$\alpha>0$
(sometimes called the `intensity parameter') is related to the variance of the branching mechanism,
 and $D$ is the region where the underlying spatial motion lives. 
 (A more general branching mechanism, including an integral term, corresponding to infinite variance, was introduced by E. B. Dynkin, but we do not work with those branching mechanisms in this paper.)

The idea behind the notion of superprocesses can be traced back to W. Feller, who observed in his 1951 paper on diffusion processes in genetics, that for large populations one can employ a model obtained from the Galton-Watson process, by rescaling and passing to the limit. The resulting {\it Feller diffusion} thus describes the scaling limit of the population mass. This is essentially the idea behind the notion of {\it continuous state branching processes.} They can be characterized as $[0,\infty)$-valued Markov processes, having paths which are right-continuous with left limits, and for which the corresponding probabilities $\{ P_x, x\ge 0\}$ satisfy the branching property: the  distribution of the process at time $t\ge 0$ under $P_{x+y}$ is the convolution of its distribution under $P_x$ and its distribution under $P_y$ for $x,y\ge 0$. Note that Feller diffusions focus on the evolution of the {\it total mass} while ignoring the location of the individuals in
the population.  The  first person who studied continuous state branching processes was  the Czech mathematician M. Ji\v{r}ina in 1958 (he called them `stochastic branching processes with continuous state space'). 

When the {\it spatial motion} of the individuals is taken into account as well, one obtains a scaling limit which is now a measure-valued branching process, or superprocess. The latter name was coined by E. B. Dynkin in the 1980's. Dynkin's work (including a long sequence of joint papers with S. E. Kuznetsov) concerning superprocesses and their connection to nonlinear partial differential equations was ground breaking. These processes are also called {\it Dawson-Watanabe processes} after the fundamental work of S. Watanabe \cite{Wat} in the late 1960's (see also the independent work by M. L. Silverstein \cite{Si}
at the same time)
and of D. Dawson \cite{Daw} in the late 1970's. Among the large number of contributions to the superprocess literature we just mention the `historical calculus' of E. Perkins, the `Brownian snake representation' of J.-F. LeGall,  the `look down construction' (a countable representation) of P. Donnelly and T. G. Kurtz, and the result of R. Durrett and E. Perkins showing that for $d\ge 2$, rescaled contact processes converge to super-Brownian motion. In addition, {\it interacting superprocesses} and {\it superprocesses in random media} have been studied, for example, by  D. Dawson, J-F. Delmas, A. Etheridge, K. Fleischmann, H. Gill, P.  M\"orters, L. Mytnik,  Y. Ren, R. Song,   P. Vogt, J. Xiong, and H. Wang, as well as by the authors of this article.

\subsection{Motivation}
A natural and interesting question in the theory of superprocesses is how fast the total mass and local mass grow as time evolves.
When $\beta$ is bounded from above (or more generally, when  $\lambda_c$, the generalized principal eigenvalue\footnote{For the definition and properties of $\lambda_c$ see Chapter 4 in \cite{Pinskybook}.} of $L+\beta$ on $D$ is finite), the problem of the local growth has been settled (see e.g. \cite{Jancsikonyv} and the references therein) and it is known that  the growth rate is at most exponential.  

The local and the global growth are not necessarily the same. In fact, another quantity, denoted by $\lambda_{\infty}$ is the one that gives the rate of the {\it global} exponential growth, when it is finite. It may coincide with $\lambda_c$ or it may be larger. Under the so-called {\it Kato-class assumption}
on $\beta$, it is finite. See subsection 1.15.5 in \cite{Jancsikonyv} for more explanation. 

In general, the growth rates of the superprocess can be super-exponential, and up to now, very little is known about the exact growth rates then. 
It is important to point out that in the general  case, even the {\it existence} of superdiffusions needs to be justified. 
The difficulty with the construction in such a situation (i.e. when $\lambda_c=\infty$) is compounded by the fact that in the lack of positive harmonic functions (i.e. functions that satisfy $(L+\beta-\lambda)u=0$ with some $\lambda$), all the usual machinery of martingales, Doob's $h$-transforms, semigroup theory etc. becomes unavailable. (When $\sup_{x\in \R^d} \beta (x) =\infty$ but $\lambda_c<\infty$, one can actually reduce the construction to the case when $\beta$ is a constant, see p.88 in \cite{Jancsikonyv}.)   
New ideas and approaches are needed for the construction and growth rate estimates.

Obtaining the precise growth rate of superprocesses with  `large' mass creation turns out to be  quite a challenging question, and there are many possible scenarios, depending on how large $\beta$ is; see Theorem \ref{T:1.1} below for example.  
The main part of  this paper is devoted to address this  question for a class of superprocesses with large mass creation. 
The effective method of `lower and upper solutions' for the partial differential equations associated with superprocess through the log-Laplace equation in the study of exponential growth rate for superdiffusions with bounded mass creation term $\beta$  becomes intimidatingly difficult if not impossible when $\beta$ is unbounded. 
(For a beautiful application of lower and upper solutions see \cite{Pinsky1995KPP,Pinsky1995Growth}.) 

In Section 5 of this article, we are going to introduce the new concept of the `$p$-generalized principal eigenvalue,' in an effort to capture the super-exponential growth rate of superprocesses with large mass creation term. 

In the last part of this paper, we will employ the `Poissonization' method to study super-exponential growth rate for superprocesses
with large mass creation, by relating them to discrete branching particle systems. 
In order to do this, we extend some results of Fleischmann and Swart given in \cite{FleischmannSwart} concerning the coupling of superprocesses and discrete branching particle systems, from deterministic times to  stopping times. This part may be of independent interest. An advantage of this method over the use of test functions, even in the case of $\lambda_c<\infty$,  is that it enables one to transfer results directly from the theory of branching diffusions, where a whole different toolset is available   as one is working with a discrete system. (A remark for the specialist: classical `skeleton decompositions' work only one way a priori, as the skeleton is a non-trivial part of the measure-valued process, after conditioning on survival. The `Poissonization' method we use, however, always works both ways.)

 Here is one of the main results of this paper, which gives the connection between the growth rate of superdiffusions and that of the corresponding branching processes.

 \begin{thm}[General comparison between $Z$ and $X$]\label{GCT}
Let $(X,P_0)$ be the superprocess corresponding to the operator $Lu+\beta u-\beta u^2$ on $D$ and $(Z,\mathsf{P}_0)$  the branching diffusion on $D$ with branching rate $\beta$,  started at the origin with unit mass, and with a {\sf Poisson(1)} number of particles, respectively. Let $|X|$ and $|Z|$ denote the total mass processes. Denote by  
$$
S:=\{|X_t|>0 \hbox{ for every }  t\ge 0\}
$$
 the event of survival for the superdiffusion.
 Let $f:[0,\infty)\to[0,\infty)$ be a continuous function such that $\lim_{x\to\infty}f(x)=\infty$.
Then 
\begin{enumerate}
\item[(i)] the condition
\begin{equation}\label{feltetel}
\mathsf{P}_0\left(\limsup_t \frac{|Z_t|}{f(t)}\le 1 \right)=1
\end{equation}
 implies that
 $P_0\left(\limsup_t \frac{|X_t|}{f(t)}\le 1\right)=1;$
 \item[(ii)]
 the condition
 \begin{equation}\label{feltetel2}
\mathsf{P}_0\left(\liminf_t \frac{|Z_t|}{f(t)}\ge 1 \right)=1
\end{equation}
 implies that
 $P_0\left(\liminf_t \frac{|X_t|}{f(t)}\ge 1 \mid S\right)=1,$
provided that one has \linebreak $P_0(\lim_{t\to\infty}|X_t|=\infty\mid S)=1$. This latter condition is always satisfied if the coefficients of $\frac{1}{\beta} L$ are bounded from above.
 \end{enumerate}
\end{thm}

Using Theorem \ref{GCT} and the results from  \cite{BBHH2010,BBHHM2015} on the corresponding branching Brownian motions, we have the following result, which  illustrates  some possible super-exponential 
growth rates the total mass of  a super-Brownian motion with large mass creation term $\beta$ may  have.

\begin{thm}\label{T:1.1}
Let $X$ be a one-dimensional super-Brownian motion corresponding to $(\frac12 \Delta,  \beta, \beta; \R)$.
Let $S$ be as in Theorem \ref{GCT}.
Then
 \begin{enumerate}
\item If $\beta(x)=1+|x|^p $ for  $0\leq p<2$,  then
 $$
 \lim_{t\to\infty}\frac{1}{t^{\frac{2+p}{2-p}}}\log |X_t|= K_p  \quad P_0-\hbox{a.s. on }S,
 $$ 
 where $K_p$ is positive constant, depending on $p$.

\item If $\beta(x)=1+C|x|^2$,  with $C>0$, then
$$
\lim_{t\to\infty}(\log\log |X_t|)/t= 2\sqrt{2}C  \quad P_0-\hbox{a.s. on } S. 
$$
\end{enumerate}
\end{thm}
Note: it is not difficult to show that the survival set $S$ is not-trivial. In fact, $P_{\delta_x} (S)\geq e^{-1}$ for every $x\in \R$; see Section 6.3.

\subsection{Outline} 
The rest of the paper is organized as follows.
Section \ref{S:2} gives some preliminaries including notation that will be used later in the paper.
The first main result of this paper, regarding the construction of superdiffusions with general large mass creation, is given
in Section \ref{S:3}.  When the generalized principal eigenvalue $\lambda_c$ of $L+\beta$ on $D$ is infinite, we show in Section \ref{S:4} that the local mass of the superprocess can no longer grow at an exponential rate: the growth will be `super-exponential.' In Section \ref{S:5} we will focus on super-Brownian motion on $\R^d$ with mass creation $\beta (x)=a |x|^\ell$ for $0\le\ell \le 2$; construction and some basic properties are discussed, in particular, the  growth of the total mass for the case when $d=1$.  

We then introduce a new notion we dubbed the `$p$-generalized principal eigenvalue' (a notion more general than $\lambda_c$). Some of its properties are investigated in the Appendix.  

Section \ref{S:6} is devoted to employing a `Poissonization' method to obtain precise growth rate for the total mass of the superprocess from that of the total mass   
of the corresponding discrete branching process; see Theorem \ref{GCT}. 
The proof of  Theorem \ref{T:1.1} is given at the end of Section \ref{S:6} as a corollary to Theorem \ref{GCT}.

 \section{Preliminaries}\label{S:2}

\subsection{Notation}\label{notat}

For convenience, we first recall basic notation. Let $d\ge 1$ and $D\subseteq\mathbb{R}^{d}$
be a domain and let ${\mathcal{B}}(D)$ denote the Borel sets of $D$.
We write ${\mathcal{M}}_{f}(D)$ and ${\mathcal{M}}_{c}(D)$ for the
class of finite measures and the class of finite measures with
compact support on ${\mathcal{B}}(D)$, respectively,  and $\mathcal{M}_{loc}(D)$
denotes the space of locally finite measures on ${\mathcal{B}}(D)$.
For $\mu\in{\mathcal{M}}_{f}(D)$, denote $|\mu|:=\mu(D)$ and let
$B^+_b(D),\ C_{b}^{+}(D)$ and $C_{c}^{+}(D)$ be the class of non-negative bounded
Borel measurable, non-negative bounded continuous and non-negative continuous functions $D\rightarrow\mathbb{R}$
having compact support, respectively. 
For integer $k\geq 0$ and $0\leq \eta <1$, we use $C^{k,\eta}(D)$ to denote the space
of continuous functions on $D$ that have continuous derivatives up to and including the $k$th order and
whose $k$th oder partial derivatives are $\eta$-H\"older continuous in $D$.   
We write $C^k(D)$ for $C^{k, 0}(D)$ and  $C^{\eta}(D)$ for $C^{0,\eta}(D)$.

For two nonempty sets  $D_1$ and $D_2$ in $\mathbb R^d$, the notation $D_1\Subset D_2$ will mean that $\overline{D_1}\subset D_2$ and $D_1$ is bounded.

The notation $\mu_{t}\stackrel{v}{\Rightarrow}\mu$ ($\mu_{t}\stackrel{w}{\Rightarrow}\mu$) will be used for the vague (weak) convergence of measures.

Let $\LL $ be an elliptic operator on $D$ of the form \[
\LL :=\frac{1}{2}\nabla\cdot a\nabla+b\cdot\nabla,\]
 where $a_{i,j},b_{i}\in C^{1,\eta}(D)$, $i,j=1,...,d$, for some
$\eta\in(0,1]$, and the matrix $a(x):=(a_{i,j}(x))$ is symmetric
and positive definite for all $x\in D$. In addition, let $\alpha,\beta\in C^{\eta}(D)$,
with $\alpha>0$.
  
Let   $Y=\left\{ Y_t, t\geq 0, \mathbb{P}_{x},\ x\in D\right\} $
be the {\it minimal diffusion process} in $D$ having infinitesimal generator  $\LL  $ in $D$; that is, $Y$ is a diffusion process having infinitesimal generator $\LL$ with killing upon exiting $D$.
Note that typically $Y$ may have finite lifetime $\zeta$  and thus $\mathbb{P}_{x}\left(Y_{t}\in D\right)\le1$ in general.
 (In the terminology of \cite{Pinskybook}, $Y$ is the solution of the {\it generalized martingale problem} for $L$ on $D$. The world `generalized' refers to the fact that conservativeness is not assumed.)
Finally, let
 \begin{eqnarray*}
\lambda_c &=&\lambda_{c}(\LL +\beta,D) \\
&:=& \inf\{\lambda\in\mathbb{R}\ :\exists   u\in C^2\hbox{  with }u>0,\ (\LL +\beta-\lambda)u=0\ \text{in}\ D\}
\end{eqnarray*}
denote the \textit{generalized principal eigenvalue} for $\LL +\beta$ on $D$.
See Section 4.3 in \cite{Pinskybook} for more on this notion, and on its relationship with $L^2$-theory. (Here the word `generalized' basically refers to the fact that $L$ is not necessarily self-adjoint.)

\subsection{The construction of the  $(\LL ,\beta,\alpha;D)$-superdiffusion}

In \cite{AOP99}  the $\mathcal{M}_{loc}(D)$-valued
$(\LL ,\beta,\alpha;D)$-superdiffusion $X$ corresponding to the semilinear
elliptic operator $\LL u+\beta u-\alpha u^{2}$ has been constructed, under the assumption that
  \begin{equation}
\lambda_{c}(\LL +\beta,D)<\infty.\label{gpe.finite}\end{equation}
For the case when $\beta$ is upper bounded, the construction of
an $\mathcal{M}_{f}(D)$-valued process relied on the method of Dynkin
and Fitzsimmons \cite{Dynkin2002book,Fitzsimmons,Fitzsimmonscorrection},
but instead of the mild equation,  the strong equation (PDE)
was used in the construction. (In \cite{Dynkin2002book} $\beta$ is assumed
to be bounded from above and also below. This is related to the fact
that the mild equation is used in the construction.) Then
a nonlinear $h$-transform (producing `weighted superprocesses') has been  introduced in \cite{AOP99}, and with the help of this transformation
it became possible to replace $\sup_{D}\beta<\infty$ by (\ref{gpe.finite})
and get an $\mathcal{M}_{loc}(D)$-valued process. The condition (\ref{gpe.finite})
is always satisfied when $\beta$ is bounded from above, and in many
other cases as well (for example on a bounded domain $\beta$ can
be allowed to blow up quite fast at the boundary --- see p. 691 in \cite{AOP99}).

Nevertheless, (\ref{gpe.finite}) is often very restrictive. For example,
when $\LL $ on $\mathbb{R}^{d}$ has constant coefficients, then even
a {}``slight unboundedness'' destroys (\ref{gpe.finite}), as the
following lemma shows.

\begin{lem}
\label{slight.unbddness} Assume that $\LL $ on $\mathbb{R}^{d}$ has
constant coefficients and that there exists an $\varepsilon>0$ and a
sequence $\{x_{n}\}$ in $\mathbb{R}^{d}$ such that
$$
\lim_{n\to\infty}\inf_{x\in B(x_{n},\varepsilon)}\beta(x)=\infty.
$$
   Then  \eqref{gpe.finite}  does not hold for $D=\R^d$.

 \end{lem}

\begin{proof}
By the assumption, for every $K>0$ there exists an $n=n(K)\in\mathbb{N}$
such that $\beta\ge K$ on $B_{\varepsilon}(x_{n})$. Let $\lambda^{\varepsilon}$
denote the principal eigenvalue of $\LL $ on a ball of radius $\varepsilon$.
(Since $\LL $ has constant coefficients, $\lambda^{\varepsilon}$ is well
defined.) Since \[
\lambda_{c}=\lambda_{c}(\LL +\beta,\mathbb{R}^{d})\ge\lambda_{c}(\LL +\beta,B_{\varepsilon}(x_{n}))\ge\lambda^{\varepsilon}+K,\]
 and $K>0$ was arbitrary, it follows that $\lambda_{c}=\infty$.
\end{proof}
\noindent The first purpose of this paper is to replace (\ref{gpe.finite})
by a much milder condition. We note that for the discrete setting (branching diffusions), super-exponential growth has been studied in \cite{HH2009,BBHH2010,BBHHR2015}. In the recent paper  \cite{EKW2015} the connection between the two types of processes has been studied. 

\subsection{Condition replacing (\ref{gpe.finite})}
Recalling that $Y$ is the diffusion
process corresponding to $\LL $ on $D$ with lifetime
$\tau_{D}:=\inf\{t\ge0\mid Y_{t}\not\in D\}\in(0,\infty]$, 
let us define $\{T_{t},t\ge0\}$,  the formal\footnote{Finiteness, continuity or the semigroup property are not required, hence the adjective.} `Dirichlet-Schr\"odinger semigroup'
of $\LL +\beta$ in $D$, by   
\[
(T_{t}g)(x):=\mathbb{E}_{x}\left[\exp\left(\int_{0}^{t}\beta(Y_{s})
\,\mathrm{d}s\right) g(Y_{t}); t<\tau_{D}\right]\in[0,\infty],
\]
when $g\in C^{+}(D)$, $t\ge0$, and $x\in D$.

\medskip

The following  assumption, requiring that $T_t(h)$ is finite for all times for just a single positive function,  will be crucial in the construction of the superprocess.

\begin{assumption}[Existence of $T$ for a single $h>0$]\label{main.assumption}
Assume that there exists a positive function $h\in C^{2}(D)$ satisfying
that $T_{t}h (x)<\infty$ for each $t>0$ and $x\in D$.
\end{assumption}

Let  $C_c^{2,\alpha}(D):=C_c(D)\cap C^{2,\alpha}(D)$.

\begin{proposition}[Equivalent formulation]
Assumption \ref{main.assumption} is  equivalent to the following condition:
For some (or equivalently, all) non-vanishing $0\le \psi\in C_c^{2,\alpha}(D)$,  $T_t \psi<\infty$ for all $t>0$.  
\end{proposition}
\begin{proof}
It is sufficient to show that 

{\sf (a)} if for some non-vanishing $0\le \psi\in C_c^{2,\alpha}(D)$,  $T_t \psi<\infty$ for all $t>0$, then Assumption \ref{main.assumption} holds;

{\sf (b)} if for some non-vanishing $0\le \psi\in C_c^{2,\alpha}(D)$ and some $x_0\in D, t>0$, we have $T_t (\psi)(x_0)=\infty$,
then Assumption \ref{main.assumption} fails.

Indeed, in the first case, for every $x\in D$,
\begin{equation*}
h(x):=T_1\psi (x) =
\E_x \left[ e^{ \int_0^t \beta(Y_s)}\, \psi   (Y_t); t<\tau_D \right] >0,
\end{equation*}
because $$\P_x \left[ e^{ \int_0^t \beta(Y_s)}\, \psi   (Y_t)>0 \mid t<\tau_D \right] >0,$$
as $e^{ \int_0^t \beta(Y_s)}>0$ and $\P_x \left[\psi   (Y_t)>0 \mid t<\tau_D \right] >0.$
Then, clearly, $T_t h=T_{t+1}\psi<\infty$ for all $t>0$, and thus Assumption \ref{main.assumption} will hold.

In the second case, for any  $0<h\in C^2(D)$, there exists a $C>0$ such that $Ch>\psi$, implying
$T_t (h)(x_0)=\infty$, and thus Assumption \ref{main.assumption} cannot hold.
\end{proof}

\begin{rem}\label{R:2.4}
Approximating $D$ by 
an increasing sequence of relatively compact 
domains and using standard compactness arguments, it is not difficult to show that under Assumption \ref{main.assumption}, the function $u$ defined by $u(x, t):= T_t h (x)$ solves the parabolic equation
$$
\frac{\partial u}{\partial t} = (\LL +\beta) u\hspace{1cm}
  \mbox{in }D\times (0,\infty),
$$
 and in particular, $u \in C([0, \infty) \times D)$.  
\end{rem}

\medskip
 
 \subsection{A useful maximum principle}

 In the remaining part of this paper, for convenience, we will use either $\dot u$ or $\partial_t u$ to denote $\frac{\partial u}{\partial t}$.
 We will frequently refer to the following parabolic semilinear maximum principle due to  R. Pinsky    \cite[Proposition 7.2]{AOP99}:

\begin{proposition}[Parabolic semilinear maximum principle]\label{Parabolic semilinear maximum principle} 
Let $\LL ,\beta$ and $\alpha$ be as in Subsection \ref{notat} and let 
 $U \Subset D$ be a non-empty domain.  Assume that  the functions $0 \le v_1,v_2 \in C^{2,1}(U \times (0,\infty)) \cap C(\overline {U} \times (0,\infty))$ satisfy 
$$
\LL v_1 + \beta v_1 -\alpha v_1^2 - \dot v_1 \le \LL v_2 +\beta v_2 -\alpha v_2^2 -\dot v_2
\quad \hbox{in }  U \times (0,\infty), 
$$
  $v_1(x,0) \ge  v_2(x,0)$ for $x \in U$, and $v_1(x,t) \ge v_2(x,t)$ on  $\partial U\times (0, \infty)$. Then $v_1 \ge v_2$ in $U\times[0,\infty)$. 
\end{proposition}

\section{Superprocess with general mass creation}  \label{S:3}

The following theorem is one of the main results of this paper, on  the construction of the superprocess with large mass creation. 

\begin{thm}[Superprocess with general mass creation]   \label{main.thm} 
Under Assumption \ref{main.assumption}  there exists a unique $\mathcal{M}_{loc}(D)$-valued Markov process $\{\left(X,P_{\mu}\right);\mu\in \mathcal{M}_c(D)\}$  satisfying the log-Laplace equation
\begin{equation}
E_{\mu}\exp(\langle-g,X_{t}\rangle)=\exp(\langle-S_{t}(g),\mu\rangle),\ g\in C_{c}^{+}(D),\ \mu\in\mathcal{M}_{c}(D),\label{log.Laplace}\end{equation}
where $S_t(g)(\cdot)=u(\cdot, t)$ is the minimal nonnegative solution to the semilinear initial value problem (``cumulant equation")
\begin{equation}\label{cum.eq}
\left\{ \begin{array}{rll}
  \dot{u}&=\LL u+\beta u-\alpha u^2\hspace{1cm}
  \mbox{in }D\times (0,t),
, \\
 \lim_{t\downarrow 0}u(\cdot,t)&=g(\cdot).
\end{array}\right.
\end{equation}
\end{thm}

\begin{defn}
The process $X$ under the probabilities $\{P_{\mu},\ \mu\in\mathcal{M}_{c}(D)\}$
in Theorem \ref{main.thm} will be called the \emph{$(\LL ,\beta,\alpha;D)$-superdiffusion}. 
\end{defn}

\begin{rem}
 (i)  Although we only consider the operator $\LL u+\beta u-\alpha u^{2}$ 
in this paper, the construction of the superprocess goes through for
the operator $\LL u+\beta u-\alpha u^{1+p}$, $0<p<1$,  as well.

(ii)  
Condition (\ref{gpe.finite})  implies Assumption
\ref{main.assumption}. This is because if (\ref{gpe.finite}) holds, 
 then there is  a $C^2(D)$-function   $h>0$ such that $(\LL + \beta -\lambda_c ) h=0$ in $D$.
(See Section 4.3 in \cite{Pinskybook}.)  Clearly, it is enough to  show that $T_t(h)\le h.$

Let $\{D_k; k\geq 1\}$ be an increasing sequence of relatively compact smooth subdomains of $D$ with $D_k\Subset D_{k+1}\Subset D$ that increases to $D$.
By the Feynman-Kac representation, for every $k\geq 1$, 
$$ 
0\le u^{(k)}(x,t):=\E_x \left[ e^{\int_0^t [\beta(Y_s)-\lambda_c] \mathrm{d}s} h(Y_t); t<\tau_{D_k} \right]   ,  \quad x\in D_k, t\geq 0,
$$
is the unique parabolic solution for  $\dot{u}=(L+\beta-\lambda_c)u$ on $D_k$ with zero boundary condition and initial condition $h$.

By taking $k\to \infty$, and using the above Feynman-Kac representation (or the parabolic maximum principle), $u^{(k)}$ are monotone nondecreasing in $k$, and are all bounded from above by $h$ (which itself is a nonnegative parabolic solution on each domain $D_k$ with initial condition $h$ restricted on $D_k$).
Therefore, by the Monotone Convergence Theorem, the limiting function $u$ satisfies that
 $$ 
h(x)\ge u(x,t)= \E_x \left[ e^{\int_0^t [\beta(Y_s)-\lambda_c] \mathrm{d}s} h(Y_t); t<\tau_{D } \right]=T_t(h)(x).
\hspace{2cm}\diamond $$
\end{rem}

From (\ref{log.Laplace}) it follows that $X$ possesses the \textit{branching
property.}
\begin{cor}[Branching property]
 If \textup{$\mu,\nu\in\mathcal{M}_{c}(D),\ t\ge0$ and $g\in C_{c}^{+}(D),$
then the distribution of $\langle g,X_{t}\rangle$ under $P_{\mu+\nu}$
is the convolution of the distributions of $\langle g,X_{t}\rangle$ under
$P_{\mu}$ and under $P_{\nu}$. }
\end{cor}
We first recall the definition
  of the nonlinear space-time $H$-transform. Consider the backward
operator
$$
\begin{aligned}
   {\mathcal A}(u)&:=\partial_s
   u+({\LL }+{\beta})u-{\alpha}u^2,
\end{aligned}
$$
and let $0<H\in 
 C^{2,1, \eta}(D \times \R^+)$.
 Analogously to  Doob's $h$-transform for linear operators, introduce the new operator $\mathcal{A}^{H}(\cdot):=\frac{1}{H}\mathcal{A}(H\cdot).$ Then a direct computation gives that
\be{H5}
\begin{aligned}
   {\mathcal A}^H(u)=
   \frac{\partial_sH}{H}u+\partial_su+{\LL }u+{a}\frac{\nabla H}{H}\cdot
   \nabla u+{\beta}u+\frac{{\LL }H}{H}u-{\alpha}Hu^2.
\end{aligned}
\end{equation}

This transformation of operators has the following probabilistic impact.
Let  $X$ be a
$({\LL },\beta,\alpha;D)$-superdiffusion. We
define a new process $X^H$ by
\begin{equation}\label{newprocess}
  X_t^H:=H(\cdot,t)X_t\quad \left(\mbox{that is,}\
  \frac{\mathrm{d}X_t^H}{\mathrm{d}X_t}=H(\cdot,t)\right), \quad
  t\geq 0.
\end{equation}
 In this way, 
one obtains a new superdiffusion, which, in general, {\it
is not finite measure-valued} but only $\mathcal{M}_{loc}(D)$-valued.
The connection between $X^H$ and $\mathcal{A}^H$ is given by the following result.
\beL{LH1}{Lemma 3 in \cite{ew2006}}

The process $X^H$, defined by (\ref{newprocess}),
is a superdiffusion corresponding to $\mathcal{A}^H$ on $D$.
\end{lemma}
Note that the differential operator $\LL $  is transformed into 
$$
\LL_0^H:={\LL }u+{a}\frac{\nabla H}{H}\cdot\nabla,
$$
 while $\beta$ and $\alpha$ transform into 
 $\beta^H:={\beta}+\frac{( \partial_s  + \LL) H}{H} $
and $\alpha^H:=\alpha H$, respectively.

It is clear that given a superdiffusion,  $H$-transforms can be used to produce new superdiffusions that are weighted versions of the old one. 
 See \cite{ew2006} for more on $H$-transforms. 
 We now show that, under the assumption of Theorem \ref{main.thm}, one can always use $H$-transforms to construct the superdiffusion.

Recall that by Assumption  \ref{main.assumption}, there exists an $h>0$ such that $(T_t h)(x)<\infty$ for all $t\ge 0$ and $x\in D$. Let us fix such an $h$.
We first work with a {\it fixed finite time horizon}. Fix $t>0$ and for $x\in D,r\in[0,t]$, consider
 $$
 H(x,r;t,h):=(T_{t-r}h)(x)<\infty.
 $$
 Then $0<H\in C^{2,1, \eta}(D \times \R^+)$ and $H$ solves the backward equation
\begin{equation}\label{backward}
\begin{aligned}
   -\partial_r H&={\LL }H+{\beta} H\hspace{1cm}
  \mbox{in }D\times (0,t),
  \\
   \lim_{r\uparrow t}H(\cdot,r;t,h)&=h(\cdot).
\end{aligned}
\end{equation}
(One can approximate $D$ by an increasing sequence of compactly embedded domains $D_n$ and consider the Cauchy problem with Dirichlet boundary condition. By the maximum principle, the solutions are growing in $n$, and, by the assumption on $h$, the limit is finite. That the limiting function  is a solution and it belongs to $C^{2,\eta}(D)\times C^{1,\eta}(\mathbb R^+)$, follows by using standard {\it a priori} estimates and compactness in the second order H\"older norm; see Theorems 5 and 7 in Chapter 3 in \cite{FriedmanParabolicBook}.)

For the rest of this subsection fix a measure $\mu\in\mathcal{M}_c(D)$. Keeping $t>0$ still fixed, we  first   show that the (time-inhomogeneous) critical measure-valued process $\widehat X$ corresponding to the quadruple 
 $$ 
\left( L_0^H, \beta^H, \alpha^H; D \right)= \Big(L+a\frac{\nabla H}{H}\cdot \nabla, \, 0, \, \alpha H; \, D \Big)
$$
 on the time interval $[0,t]$ is well defined. To check this, recall the construction in Appendix A in \cite{AOP99}. That construction goes through for this case too, despite the time-dependence of the drift coefficient of the diffusion and the variance term $\alpha$. Indeed, the first step in the construction of the measure-valued process is the construction of the minimal nonnegative solution to the semilinear parabolic Cauchy problem (\ref{cum.eq}). It is based on the approximation of $D$ with compacts $D_n\Subset D,\cup_{n=1}^{\infty} D_n=D$, and imposing zero Dirichlet boundary condition on them (see the Appendix A in \cite{AOP99}). By the local boundedness of $\beta$, the solution with zero boundary condition for the original operator is well defined on compacts, and therefore it is also well defined  for the $H$-transformed operator on compacts. As $n\to\infty$, the solution to this latter one does not blow up, because the new potential term is zero and because of Proposition \ref{Parabolic semilinear maximum principle}. Hence, the solution to the original Cauchy problem does not blow up either.

Once we have the minimal nonnegative solution to the $H$-transformed Cauchy problem
 we have to check that it defines, via the log-Laplace equation, a finite measure-valued Markov process on the time interval $[0,t]$.

Let $S_s^H(g)(x):=u^{(g)}(x,s)$, where $u^{(g)}$ denotes the minimal nonnegative solution to the $H$-transformed nonlinear
Cauchy problem 
 $$
    \dot{u} =\LL_0^H u -\alpha^H  u^2 
      $$
 with $u(x,0)=g(x)\in C_{b}^+(D)$. Note that
\begin{equation}\label{down}
S_{s}^{H}(g_{n})\downarrow0\ \text{pointwise, whenever }
 g_n 
\in C_{b}^{+}(D),\ \text{and}\ g_{n}\downarrow0\ \text{pointwise},
\end{equation}
because, using the semilinear parabolic
maximum principle and the fact that 
 $$
S_{s}^{H}(g_{n})\le T_{s}^{H}(g_{n})\le \| g_n \|_\infty,
$$
 where $\{T^H_s; s\geq 0\}$ is the semigroup associated with the infinitesimal generator $\LL^H$
with Dirichlet boundary condition on $\partial D$.
 This also shows that the shift $S_{t}^{h}$ leaves $C_{b}^{+}(D)$ invariant.

Before proceeding further, let us note that, by the minimality of the solution
$S^{H}$ forms a semigroup on $C_b^+(D)$:
\begin{equation}\label{semigroup}
S^H_{s+z}=S^H_s\circ S^H_z,\ \text{for}\  0\le s,z\ \text{and}\  s+z\le t.
\end{equation}
(Obviously, $S_0$ is the unit element of the semigroup.)

Reading carefully the construction in \cite{Dynkin2002book,DynkinAnnals1993paper}
along with the one in Appendix A of \cite{AOP99}, one can see that
in order to define the $\mathcal{M}_{f}(D)$-valued superprocess $\widehat{X}$
corresponding to $S^{H}$ (on $[0,t]$) via the log-Laplace equation \begin{equation}
\E_{\mu}^{H}\exp(\langle-g,\widehat{X}_{s}\rangle)=\exp(\langle-S_{s}^{H}(g),\mu\rangle),\ g\in C_{b}^{+}(D),\,\mu\in\mathcal{M}_{f}(D),\label{h-transformed.log-Lapl.}\end{equation}
 one only needs that $S^{H}$ satisfies (\ref{down}) and (\ref{semigroup}). In particular, property (\ref{semigroup})
for $S^{h}$ guarantees the Markov property for the superprocess $\widehat{X}$.

Below we sketch how this construction goes, following Appendix A in
\cite{AOP99}. The fundamental observation is that $S^{H}$ enjoys
the following three properties:
\begin{enumerate}
\item $S_{s}^{H}(0)=0;$
\item The property under (\ref{down});
\item $S_{s}^{H}$ is an \textit{N-function} on $C_{b}^{+}(D)$; that is\footnote{An explanation of the terminology `P-function' and `N-function'
is given on pp. 40-41 in \cite{Dynkin2002book}. Note that in \cite{AOP99}
we used the names positive semidefinite and negative semidefinite,
respectively.},
\[
\sum_{i,j=1}^{n}\lambda_{i}\lambda_{j}S_s^H(f_{i}+f_{j})\le0\ \text{if}\ \sum_{i}^{n}\lambda_{i}=0,\ \forall n\ge2,\ \forall f_{1},...,f_{n}\in C_{b}^{+}(D).\]

\end{enumerate}
For the third property, just like in \cite{AOP99}, one utilizes \cite{DynkinAnnals1993paper}
(more precisely, the argument on p. 1215).

Then, one defines $\mathcal{L}_{s}^{H}(\cdot):=\exp(-S_{s}^{H}(\cdot))$,
$0\le s\le t$ on $C^{+}(D)$, and checks that it satisfies
\begin{enumerate}
\item $\mathcal{L}_{s}^{H}(0)=1;$
\item $\mathcal{L}_{s}^{H}g\in(0,1]$ for $f\in C^+(D)$;
\item The property under (\ref{down}), if decreasing sequences are replaced
by increasing ones;
\item $\mathcal{L}_{s}^{H}$ is a \textit{P-function} on $C_{b}^{+}(D)$
; that is, \[
\sum_{i,j=1}^{n}\lambda_{i}\lambda_{j}\mathcal{L}_{s}^{H}(f_{i}+f_{j})\ge0,\ \forall n\in\mathbb{N},\ \forall f_{1},...,f_{n}\in C_{b}^{+}(D),\ \forall\lambda_{1},...,\lambda_{n}\in\mathbb{R}.\]

\end{enumerate}
(For the fourth property, see p. 74 in \cite{BCR}.) As noted in \cite{AOP99},
these four properties of $\mathcal{L}^{H}$ imply that for every $x\in D$
and $0\le s\le t$ fixed, there exists a unique probability measure $\widehat{P}^{x,s}$
on $\mathcal{M}_{f}(D)$ satisfying for all $g\in C_{b}^{+}(D)$ that\[
\mathcal{L}_{s}^{H}(g)(x)=\int_{\mathcal{M}_{f}(D)}e^{-\langle g,\nu\rangle}\:\widehat{P}^{x,s}(\text{d}\nu).\]
(As explained on p. 722 in \cite{AOP99}, one can use Corollary A.6
in \cite{Fitzsimmons} with a minimal modification. Alternatively,
use Theorem 3.1 in \cite{Dynkin2002book} instead of \cite{Fitzsimmons}.
The integral representation of $\mathcal{L}_{s}^{H}(g)(x)$ above
is essentially a consequence of the Krein-Milman Theorem, which can
be found e.g. in section 2.5 in \cite{BCR}.) It then follows from
the property under (\ref{semigroup}) that the functional $\mathcal{L}^{H}$
defined by \[
\mathcal{L}^{H}(s,\mu,g):=\exp\left(-\langle S_{s}^{H}g(x)\:\mu\rangle\right),\ g\in C_{b}^{+}(D),\quad \mu\in\mathcal{M}_{f}(D)\]
is a \textit{Laplace-transition functional,} that is, there exists
a unique $\mathcal{M}_{f}(D)$-valued Markov process $\left(\widehat{X},\widehat{P}\right)$,
satisfying that

\[
\mathcal{L}^{H}(s,\mu,g)=\widehat{\E}_{\mu} \left[ e^{-\langle g,\widehat{X}_{s}\rangle} \right],\quad s\ge 0,g\in C_{b}^{+}(D),\ \mu\in\mathcal{M}_{f}(D),\]
finishing the construction of $\widehat{X}.$

Now consider $\widehat X$ corresponding to the quadruple $(L+a\frac{\nabla H}{H}\cdot \nabla,0,\alpha H;D)$ on the time interval $[0,t]$ starting with initial measure $\widehat\mu^{t,h}:=H(\cdot,0;t,h) \mu$.
By the properties of the $H$-transform reviewed above,  the measure-valued process $X_r:=H^{-1}(\cdot,r;t,h)\widehat X_r$ corresponds to the quadruple $(L,\beta,\alpha;D)$ on the same time interval $r\in[0,t]$, with initial measure $\mu$.

In other words, stressing now the dependence on $t$ in the notation, if $\widehat \P^{(t)}$ corresponds to $\widehat X^{(t)}$, then the measure valued process
$$
X^{(t)}_r:=H^{-1}(\cdot,r;t,h)\widehat X^{(t)}_r
$$
 under $\widehat \P^{(t)}_{\widehat\mu^{t,h}}$ satisfies the log-Laplace equation (\ref{log.Laplace}), and moreover, clearly,
 $\widehat \P^{(t)}_{\widehat\mu^{t,h}}(X^{(t)}_0=\mu)=1.$

This, in particular, shows that the definition is consistent, that is, if $t<t'$, then $\widehat P^{(t)}_{\widehat\mu^{t',h}}(X^{(t)}_\cdot\in \cdot)$ and $\widehat P^{(t')}_{\widehat\mu^{t,h}}(X^{(t')}_\cdot\in \cdot)$ agree on $\mathcal{F}_t$, and thus we can extend the time horizon of the process $X$ to $[0,\infty)$ and define a probability $P$ for paths on $[0,\infty)$. Indeed the finite dimensional distributions up to $t$ are determined by the same log-Laplace equation and $\widehat P^{(t')}_{\widehat\mu^{t',h}}(X^{(t')}_0=\mu)=1$ is still true when we work on $[0,t']$.

The semigroup property (or equivalently, the Markov property) is inherited from $S^H$ to $S$ (from $\widehat X$ to $X$) by the definition of the $H$-transform.

Our conclusion is that the $\mathcal M_{loc}(D)$-valued Markov process $\{\left(X,P_{\mu}\right);\mu\in \mathcal{M}_c(D)\}$ is well defined on $[0,\infty)$ by the log-Laplace equation (\ref{log.Laplace}) and the cumulant equation ($\ref{cum.eq}$).  \qed

\begin{rem} There is a similar construction in \cite{Schied1999} but under far more restrictive conditions on the function 
 $h$ than our Assumption \ref{main.assumption}. 
 $\hfill\diamond$
  \end{rem}

\begin{rem} [global supersolutions]
  If there exists an $0<H\in C^{2,\eta}(D)\times C^{1,\eta}(\mathbb R^+)$ which is a {\it global} super-solution to the  backward equation, i.e.  $$\dot{H}+(L+{\beta}) H\le 0\hspace{1cm}
  \mbox{in }D\times (0,\infty),$$ then there is a shorter way to proceed, since instead of working first with finite time horizons, one can work directly with $[0,\infty)$. Indeed,   similarly to what we have done above, the time-inhomogeneous (sub)critical measure-valued process $\widehat X$ corresponding to the quadruple $(L+a\frac{\nabla H}{H}\cdot \nabla,(\dot{H}+(L+{\beta}) H)/H,\alpha H;D)$  is well defined, because the potential term is non-positive. Just like before,  the measure-valued process $X_t:=H^{-1}(\cdot,t)\widehat X_t$ corresponds to the quadruple $(L,\beta,\alpha;D)$.

When $\lambda_c <\infty$, 
let $h>0$ be a $C^2$-function on $D$ with $(\LL +\beta ) h = \lambda h$ for some $\lambda \geq \lambda_c$.
Then $H(x,t):=e^{-\lambda  t}h(x)$
is  a global {\it solution} to the  backward equation
 in $D\times (0, \infty)$;  
when $\lambda_c=\infty$, global backward super-solution might not exists. $\hfill\diamond$
 \end{rem}
\begin{rem}

In \cite{AOP99}, instead of Property (\ref{down}), the  continuity on $C_{b}^{+}(D)$
with respect to bounded convergence was used. Clearly, if one knows that (\ref{down}) (together with the other  properties) guarantees the existence
of $\widehat{P}^{x,s}$ for all $x,s$, then this latter continuity
property will guarantee it too: if $0\le g_{n}\uparrow g$ and $g$ is
bounded, then the convergence is bounded. In \cite{AOP99},  in
fact,  the continuity of the semigroup with respect to bounded
convergence was proved.$\hfill\diamond$
\end{rem}

\section{Super-exponential growth when $\lambda_c=\infty$} \label{S:4}

When the generalized principal eigenvalue is infinite, the local mass of the superprocess  
can no longer grow at an exponential rate, as the following result shows.

\begin{thm}
Assume that $\mathbf{0}\neq\mu\in \mathcal{M}_c(D)$ and $\lambda_c=\infty$. Then, for any $\lambda\in\mathbb R$ and any open
set $\emptyset\neq B\Subset D$,
\begin{equation}\label{limsup}
P_{\mu}\left(\limsup_{t\to\infty}e^{-\lambda t}X_t(B)=\infty\right)>0.
\end{equation}
\end{thm}
\begin{proof} We are following the proof of Theorem 3(ii) in \cite{EK2004}.

We may assume without the loss of generality that $\lambda>0$. Since $\lambda_c=\infty$, by standard theory, there exists a large enough $B^*\Subset D$ with a smooth boundary so that $$\lambda^*:=\lambda_c(L+\beta,B^*)>\lambda.$$
In addition, we can choose $B^*$ large enough so that $\mathrm{supp}(\mu)\Subset B^*.$

 Let the eigenfunction $\phi^*$ satisfy $(L+\beta-\lambda^*)\phi^*=0$, $\phi^*>0$ in $B^*$ and $\phi^*=0$ on $\partial B^*.$
Let $X^{t,B^*}$ denote the exit measure\footnote{See \cite{Dynkin2002book} for more on the exit measure.} from $B^*\times [0,t)$. We would like to integrate $\phi^*$ against $X^{t,B^*}$, so formally we define
for each fixed $t\ge 0$, $\phi^{*,t}:B^*\times [0,t]\to[0,\infty)$ such that $\phi^{*, t}(\cdot,u)=\phi^* (\cdot)$ for each $u\in [0,t]$.
Then $\langle\phi^{*,t},X^{t,B^*}\rangle$ is defined in the obvious way.   Now define
$$
M_t^{\phi^*}:=e^{-\lambda^* t} \langle\phi^{*,t},X^{t,B^*}\rangle/\langle\phi^*,\mu\rangle.
$$
 Since $\lambda^* > 0$, Lemma 6 in \cite{EK2004} implies that $M^{\phi^*}_t$ is a continuous mean one $\P_{\mu}$-martingale and that $P_{\mu}\left(\lim_{t\to\infty}M_t>0\right)>0.$
Since $\phi^*\geq 1/c>0$ on $B^*$, we have
$$
X_t(B^*)\ge c\, \langle \phi^*|_{B^*},X_t\rangle\ge c \, \langle \phi^{*,t},X^{t,B^*}\rangle,\qquad \P_{\mu} \hbox{-a.s.}.
$$
Hence
\begin{align*}
  \P_{\mu}\left(\lim_{t\to\infty}e^{-\lambda t}X_t(B^*)=\infty\right)
  &\ge \P_{\mu}\left(\liminf_{t\to\infty}e^{-\lambda^* t}X_t(B^*)>0\right) \\
& \ge \P_{\mu}\left(\lim_{t\to\infty}M_t>0\right)>0.
\end{align*}
Now let $B$ be {\it any} open set with $\emptyset \neq B\Subset D$. Then (\ref{limsup}) follows exactly as in the end of the proof of Theorem 3(ii) in \cite{EK2004}, on p. 93.
\end{proof}

 The rest of the paper is to investigate the super-exponential growth rate for certain 
superprocesses  with infinite generalized principal eigenvalues.

\section{Conditions and Examples}\label{S:5}

\subsection{Brownian motion with $|x|^{\ell}$ potential}

For the next example, we will need the following result.
\begin{lem}
\label{lem:large deviation} Let $B$ denote standard Brownian motion in $\R^d$ with $d\geq 1$
 and let $\ell>0.$ Then

\begin{equation}\label{follows.from.Schilder}
\log\mathbb{P}\left(\int_{0}^{1}|B_{s}|^{\ell}\,\mathrm{d}s\ge K\right)=-\frac{1}{2}c_{\ell}K^{2/\ell}\left(1+o(1)\right),
\end{equation}
as \textbf{$K\uparrow\infty.$} Furthermore, $c_{1}=3.$ \end{lem}
\begin{proof}
First, the asymptotics (\ref{follows.from.Schilder}) follows directly by taking $\varepsilon=K^{-2/\ell}$ in Schilder's Theorem (Theorem 5.2.3 in \cite{DZbook})  and using  the Contraction Principle (Theorem 4.2.1 in \cite{DZbook}). We then get
 \[
c_{\ell}=\inf \left\{ \int_{0}^{1}|\dot{f}(s)|^{2}\,\mathrm{d}s: 
f\in C([0,1] \hbox{ with }  f(0)=0 \hbox{ and } \|f\|_\ell=1\right\} ,  
\]
where $\|f\|_\ell:=(\int_{0}^{1}|f(s)|^{\ell}\:\mathrm{d}s)^{1/\ell}$.
To determine the value
of $c_{1}$, one can utilize the results in \cite{SharpPoincare1,SharpPoincare2}:
 by taking $p:=2$ and $ p':=\frac{1}{1-\frac{1}{p}}=2$ in 
\cite[p. 2311, line -8]{SharpPoincare1}
and exploiting formula (1.7) there to show that $c_{1} =3$.  
\end{proof}

\begin{rem}
One can actually get a crude upper estimate for all $\ell>0$ without
using Schilder's Theorem 
but using the reflection principle for Brownian motion instead. For simplicity, we illustrate this for $d=1$.
Let $R_{t}:=\max_{s\in[0,t]}|B_{s}|.$   Then
$$\mathbb{P}\left(\int_{0}^{1}|B_{s}|^{\ell}\,\mathrm{d}s\ge K\right)\le \mathbb{P}\left(R_{1}^{\ell}\ge K\right)
 \leq  4\mathbb{P}\left(B_{1} \ge K^{1/\ell} \right)
 \leq \frac{4}{K}e^{-\frac{1}{2}K^{2/\ell}}.
$$
See, e.g., \cite[Theorem 1.2.3]{Durrettbook}  for the last inequality.  $\diamond$ 
\end{rem}

\begin{example}
\label{BM.with.p-potential}Let $d\geq 1$ and $L=\frac{1}{2}\Delta$, $\beta(x)=a|x|^{\ell}$ with $a,\ell>0$, and $\alpha>0$.   From
Lemma \ref{slight.unbddness}, it is clear that (\ref{gpe.finite}) will not hold, no matter how slowly $\beta$ grows. On the other hand, letting $h\equiv1$, we have the following claim.
\begin{claim}\label{finite.for.Schrodinger}
There are three cases.
\begin{itemize}
\item[(i)] If $0<\ell <2$,  then $T^{\frac{1}{2}\Delta+\beta}_t 1(\cdot)<\infty$
for every $t>0$.

\item[(ii)] If $\ell = 2$, 
then there is some function $t_0=t_0(x)$ on $\R^d$ that is bounded between two positive
constants so that $T^{\frac{1}{2}\Delta+\beta}_t1 (x) <\infty$ for every $t<t_0 (x)$ and $T^{\frac{1}{2}\Delta+\beta}_t1 (x) \equiv \infty$ for every $t>t_0 (x)$.

\item[(iii)] If $\ell > 2$,   then $T^{\frac{1}{2}\Delta+\beta}_t1 \equiv \infty$ for every $t>0$.
\end{itemize}
\end{claim}

Consequently, when $0<\ell <2$, not only the construction of the superprocess is guaranteed  by  Theorem \ref{main.thm}, but 
 in fact 
that {\it the expected total mass remains finite for all times}.

\bigskip

\noindent{\it Proof of Claim \ref{finite.for.Schrodinger}}.
Under $\P_0$, by  Brownian scaling, we have
$$ \int_0^t |B_s|^\ell \mathrm{d}s  = \int_0^1 |B_{tr}|^\ell \, t \mathrm{d}r
{\mathop {\ =\ }\limits^{ d }}
t^{1+ \ell/2 } \int_0^1 |B_r|^\ell \mathrm{d}r.
$$
Hence we have from above and \eqref{follows.from.Schilder} that
\begin{eqnarray} 
\left(T^{\frac12 \Delta +\beta}_t 1\right) (0) &=& 
\E_0 \left[ \exp \left( a \int_0^t |B_s|^\ell \mathrm{d}s\right) \right] \nonumber \\
&=& \int_1^\infty \P_0 \left( e^{a \int_0^t  |B_s|^\ell \mathrm{d}s}>x \right) \mathrm{d}x
\nonumber \\
&=& \int_1^\infty \P_0 \left( \int_0^t |B_s|^\ell \mathrm{d}s> (\log x)/a \right) \mathrm{d}x \nonumber \\
 &=& \int_1^\infty \P_0 \left( \int_0^1 |B_s|^\ell 
\mathrm{d}s> a^{-1} t^{-1- \ell/2 } \log x \right) \mathrm{d}x  \nonumber \\
&=& \int_0^\infty a t^{1+\ell/2} e^{a u  t^{1+\ell/2}   } \P_0
 \left(  \int_0^1 |B_s|^{\ell} \mathrm{d}s > u \right) \mathrm{d}u \nonumber \\
&=& \int_0^\infty  a  t^{1+\ell/2} e^{a u t^{1+\ell/2} } \left(
e^{-\frac12 c_\ell u^{2/\ell} (1+o(1))} \right) \mathrm{d}u. \label{e:5.2}
\end{eqnarray}
The claims now clearly follow from the last integral expression.

For general $x\in \R^d$, observe that
$$
 \left(T^{\frac12 \Delta +\beta}_t 1\right) (x) =
\E_x  \left[ \exp \left(\int_0^t a |B_s|^\ell \mathrm{d}s\right) \right]
= \E_0 \left[   \exp \left(\int_0^t a |x+B_s|^\ell \mathrm{d}s\right) \right],
$$
which is bounded between $c_t$ and $C_t$, where
\begin{align*} 
&c_t:=e^{-a |x|^\ell} \E_0 \left[ \exp \left(\int_0^t 2^{-\ell} a |x+B_s|^\ell \mathrm{d}s\right) \right]; \\
&C_t:=e^{2^\ell  a |x|^\ell }\E_0 \left[ \exp \left(\int_0^t 2^\ell a |B_s|^\ell \mathrm{d}s\right) \right].
\end{align*}
The claim is thus proved. \qed
\begin{rem}

The statements of Claim \ref{finite.for.Schrodinger} can be found in Sections 5.12-5.13 of  \cite{ItoMcKean}, but since they follow very easily from Lemma \ref{lem:large deviation} (which we need later anyway), we decided to present the above proof for the sake of being more self-contained.
\end{rem}

 When $\ell=1$ we have the following estimate,  which will be used later, in Example \ref{revisit}.

\begin{claim} Assume that $d=1$ and $\beta(x)=|x|$. Then 
 \begin{equation}\label{eq:two_sided_estimate}
 e^{ t^3/6}\le E_0|X_t|=(T^{\frac12 \Delta +\beta}_t 1)(0)=\mathbb{E}_{0}\exp\left(\int_{0}^{t}|B_{s}|\,\mathrm{d}s\right)\le 4 e^{t^{3}/2}.
 \end{equation}
\end{claim}

\begin{proof}
Recall that $R_t:=\max_{s\leq t} |B_s|$. By the symmetry and the reflection principle for Brownian motion, 
 $$
\P_0 (R_t>x)\leq 2 \P_0 \left(\max_{s\in [0, t]} B_s >x \right) =4\P_0 (B_t>x)
\quad \hbox{for every } x>0.
 $$ 
Hence
\begin{eqnarray*}
\mathbb{E}_{0}\exp\left(\int_{0}^{t}|B_{s}|\,\mathrm{d}s\right)
&=& \int_0^\infty \P_0 \left( \exp\left(\int_{0}^{t}|B_{s}|\,\mathrm{d}s\right) >x \right)\mathrm{d} x  \\
&=& \int_0^\infty \P_0 \left(  \int_{0}^{t}|B_{s}|\,\mathrm{d}s  > \log x \right)\mathrm{d} x  \\
&\leq & 1+ \int_1^\infty \P_0 \left( t R_t  > \log x \right)\mathrm{d} x  \\
&\leq & 1+ 4 \int_1^\infty \P_0 \left( t B_t  > \log x \right)\mathrm{d} x  \\
&=& 1+ 4 \int_1^\infty  \P_0 \left( e^{tB_t}>x \right) \mathrm{d} x \\
&\leq & 4 \int_0^\infty  \P_0 \left( e^{tB_t}>x \right) \mathrm{d} x  = 4 \E_0 e^{tB_t} = 4 e^{t^3/2}.  
\end{eqnarray*}
Here in the last inequality we used the fact that $ \P_0 (B_t\geq 0) =1/2$ and so
$\P_x \left( e^{tB_t}>x \right) \geq 1/2$ for every $0<x<1$. 
For the lower bound, note that by It\^{o}'s formula,
$$ 
\int_0^t B_s  \mathrm{d} s = tB_t -\int_0^t s   \mathrm{d} B_s = \int_0^t (t-s)  \mathrm{d}B_s,
$$
which is of centered Gaussian distribution with  variance $ t^3/3$.
Hence
$$
\mathbb{E}_{0}\exp\left(\int_{0}^{t}|B_{s}|\,\mathrm{d}s\right)
\geq \mathbb{E}_{0}\exp\left(\int_{0}^{t} B_{s} \,\mathrm{d}s\right)
= e^{t^3/6},
$$
proving the claim. 
\end{proof}
\end{example}

\begin{example}
\label{BM.with. quadratic potential}Let $L=\frac{1}{2}\Delta$, $\beta(x)=|x|^{2}$,
and $\alpha\ge \beta$.   We 
 can define the superprocess even in this case,  using an argument involving a {\it discrete} branching particle system as follows.

As noted in the proof of Theorem \ref{main.thm}, one only needs that $S=\{S_t\}_{t\ge 0}$ satisfies the semigroup property (\ref{semigroup}) on the space $C_{b}^{+}(D)$ along with condition (\ref{down}).
 We do not need to use $H$-transform in this case. 

In order to check these, along with the well-posedness  of the nonlinear initial value problem, note  that the $d$-dimensional branching Brownian motion $(Z,\mathbf{P})$ with branching rate $\beta(x)=|x|^{2}$ `does not
blow up', that is $|Z_{t}|<\infty$  for all $t>0$, a.s., although $|Z_{t}|$ has infinite expectation.
This  follows from (ii) of Claim \ref{finite.for.Schrodinger}. Indeed, write $(\mathbf{E}_x;x\in \mathbb R^d)$ for the expectation corresponding to $Z$. Then $\mathbf{E}_x |Z_t|<\infty$ for all $x\in\mathbb R^d$, if $t$ is sufficiently small. But then, by the branching Markov property,  $|Z_t|<\infty$ for {\it all} times, $P_x$-a.s.\footnote{This is the point in the argument where we benefit from turning to the discrete system.} (Cf. \cite{HH2009}.)

Next, it is standard to show  that $(Z,\mathbf{P})$ satisfies the following log-Laplace equation:
\begin{equation}\label{lL.for.discrete}
\mathbf{E}_x e^{\langle -g,Z_t\rangle}=1-u(x,t),
\end{equation}
where $u$ is the minimal nonnegative solution to the initial value problem
\begin{equation}\label{discrete.evol.eq}
\left\{ \begin{array}{rll}
 \dot{u} &=& Lu+\beta u-\beta u^2 , \\
\lim_{t\downarrow 0}u(\cdot,t)&=&1-e^{-g}(\cdot).
\end{array}\right.
\end{equation}
More precisely, one approximates $\mathbb R^d$ by an increasing sequence of compact domains $D_n$, and for each $n$, considers the initial value problem  (\ref{discrete.evol.eq}), but on $D_n$ instead of $\mathbb R^d$, and with zero boundary condition.  Using  Proposition \ref{Parabolic semilinear maximum principle}, it follows that the solutions are increasing as $n$ grows, and that their limit stays finite as $n\to\infty$, by comparison with the constant one function. It also follows by Proposition \ref{Parabolic semilinear maximum principle} that the limiting function is the {\it minimal} nonnegative solution. (To see that the limit is actually a solution, see Appendix B in \cite{AOP99}.) For each $n$, the initial-boundary value problem yields the solution that one plugs into (\ref{lL.for.discrete}), where $\mathbb R^d$ is replaced by $D_n$ and $Z$ is replaced by the branching-Brownian motion with the same rate on $D_n$ but with {\it killing} of the particles at $\partial D_n$.

Now, consider again (\ref{discrete.evol.eq}). Above we concluded that, when the initial function is  bounded from above by one, the solution does not blow up. In fact, the same argument, using Proposition \ref{Parabolic semilinear maximum principle} shows that this is true for any bounded nonnegative initial function. Indeed,  for $K>1$, the function  $h\equiv K$ is a super-solution if the initial function $g$ satisfies $g\le K$. This argument  is obviously still valid if the operator $Lu+\beta u-\beta u^2$ is replaced by $Lu+\beta u-\alpha u^2$, provided $\alpha\ge \beta$. Therefore, in this case the initial value problem  is well-posed and can be considered the cumulant equation for the {\it superprocess}.

To define the superprocess via the log-Laplace equation using the minimal nonnegative solution to this cumulant equation, we  have to check two conditions. It is easy to see that (\ref{semigroup}) is a consequence of the minimality of the solution, while
 for condition (\ref{down}) we can use the discrete branching process as follows. By  (\ref{lL.for.discrete}), condition (\ref{down}) follows by monotone convergence when $\beta=\alpha$; when $\alpha\ge\beta$,  we are done by using  Proposition \ref{Parabolic semilinear maximum principle}.
$\hfill\diamond$
\end{example}
\begin{rem} The  argument in Example \ref{BM.with. quadratic potential} shows that, in general, whenever the branching diffusion with  $L$-motion on $D$ and branching $\beta$ is well defined and finite at all times, the $(L,\beta,\alpha;D)$-superdiffusion is also well defined and $\mathcal{M}_f(D)$-valued, provided that ($\alpha>0$ and) $\alpha\ge \beta$ . $\hfill\diamond$
\end{rem}
\subsection[CSP]{The compact support property and an example}
Recall that  $X$ possesses the \textit{compact support property} if  $P(C_{s}\Subset D)=1$ for all fixed $s\ge0$, where 
\[
C_{s}(\omega):= 
\mathsf{closure}
\left(\bigcup_{r\le s}\mathrm{supp}(X_{r}(\omega))\right).\]
 In this case, by the monotonicity in $s$, there exists an $\Omega_{1}\subset\Omega$
with $P(\Omega_{1})=1$ such that for $\omega\in\Omega_{1}$,
\begin{equation}\label{CSP.all.s}
C_{s}(\omega)\Subset D   \quad \hbox{for every } s\ge 0.
\end{equation}

It is easy to see that the criterion in \cite{AOP99} (see Theorem 3.4 and its proof in \cite{AOP99}) carries through for our more general superprocesses, that is,

\begin{proposition}[Analytic criterion for CSP]\label{Cauchy.CSP}
The compact support property holds if and only if the only  
 non-negative function $u$ satisfying 
\begin{equation}\label{zero.IC}
\left\{ \begin{array}{rcc}
 \dot{u} &=& Lu+\beta u-\alpha u^2 , \\ 
\lim_{t\downarrow 0}u(\cdot,t)&=&0,
\end{array}\right.
\end{equation}
is $u\equiv 0$; equivalently, if and only if $u_{\mathrm{max}}$, the maximal solution to \eqref{zero.IC} is identically zero.
\end{proposition}
We now apply this analytic criterion to a class of superdiffusions.

\begin{claim} \label{csp.holds} Assume that $L$ is conservative on $D$, that $T^{L+\beta}_t(1)(\cdot)<\infty$ and that $\alpha \ge \beta$.   
Then the compact support property holds for $X$.
\end{claim}

\begin{rem}
Our assumption on $T^{L+\beta}$ guarantees that the superprocess is well defined. For example, by Claim \ref{finite.for.Schrodinger} this assumption is satisfied  when $L=\frac{1}{2}\Delta$ on $D=\mathbb R^d$ and $\beta (x)=|x|^p$, $0<p<2$; the same is true of course for $\beta (x)=C+|x|^p,C>0$.
\end{rem}

 \noindent {\it  Proof of Claim \ref{csp.holds}.}  
By Propositions \ref{Parabolic semilinear maximum principle}  and \ref{Cauchy.CSP}, it is enough to consider the case when $\alpha=\beta$, and show that $u_{\mathrm{max}}$ for \eqref{zero.IC} is identically zero.

Just like in Example \ref{BM.with. quadratic potential}, we are going to utilize a {\it discrete particle system}. Namely,
consider the  $(L,\beta;D)$-branching diffusion  $Z$, and let $\{\mathbf{P}_x,\mathbf{E}_x\,x\in D\}$  denote the corresponding probabilities and expectations. A standard fact,  following easily from (\ref{lL.for.discrete}) and (\ref{discrete.evol.eq}), is that  $u_{\mathrm{max}}(x,t)=1-\mathbf{P}_x(Z_t\Subset D)$.  We need to show that 
$$\mathbf{P}_x(Z_t\Subset  D)=1 \quad \hbox{for every }  x\in D \hbox{ and } t\ge 0.
$$
 But this follows from the conservativeness assumption and from $\mathbf{P}_x(|Z_t|<\infty)=1$, where the latter  follows from the expectation formula, as we even have $\mathbf{E}_x(|Z_t|)=T^{L+\beta}_t(1)(x)<\infty$ by  assumption.
 \qed 

For super-Brownian motion with quadratic mass creation we still have the compact support property.

\begin{claim}[CSP for quadratic mass creation] Let $L=\frac{1}{2}\Delta$ on $D=\mathbb R^d$ and $\alpha(x) \ge \beta(x):=|x|^2$.   
Then the compact support property holds for $X$.\label{csp.still.holds}
\end{claim}

\begin{proof}
We now show how to modify the proof of Claim \ref{csp.holds} in this case. Even though, by Claim \ref{finite.for.Schrodinger}, the assumption of Claim \ref{csp.holds} on the semigroup no longer holds, we know  that the superprocess is well defined, as shown in Example \ref{BM.with. quadratic potential}.
Furthermore, for the corresponding branching-Brownian motion, $\mathbf{P}_x(|Z_t|<\infty)=1$ is still true  -- see \cite{HH2009}.

The rest of the proof is exactly the same as in the case of Claim \ref{csp.holds}.
\end{proof}

\subsection{Semiorbits}

In this part we discuss a method which is  applicable in the absence of positive
harmonic functions too. In this part, the assumption on the power of the
nonlinearity $(p=2)$ is important as we are using the path continuity
(in the weak topology of measures).

\textbf{\noun{(i) Assume $\lambda_{c}<\infty.$}}\textbf{ }

The almost sure upper estimate on the local growth is then based on
the existence of positive harmonic functions. Indeed, let $h$ be a positive harmonic function, that is,  let $(L+\beta-\lambda_{c})h=0,\ h>0$.
(Such a function $h$ always exists; see Chapter 4 in \cite{Pinskybook}.)
Define $H(x,t):=e^{-\lambda_{c}t}h(x)$; then for $t,s>0,$ \begin{equation}
\left(T_{t}^{L+\beta}H(\cdot,t+s)\right)(x)\le H(x,s),\label{eq:supermart.prop}\end{equation}
that is, $T_{t}^{L+\beta-\lambda_{c}}h\le h$, or equivalently, 
$T_t^{(L+\beta-\lambda_{c})^{h}}1\le1$.
Here $$\{ T_t^{(L+\beta-\lambda_{c})^{h}}; t\geq 0\}$$ is the  semigroup obtained from $\{T_t^{ L+\beta-\lambda_{c}}; t\geq 0\}$ through an $h$-transform.  
 
Using the Markov
and the branching properties together with $h$-transform theory, it then immediately follows that if
$N_{t}:=\langle H(\cdot,t),X_{t}\rangle$, then $N$ is a continuous $P_{\mu}$-supermartingale
for $\mu\in\mathcal{M}_{f}(D)$ (where $P_{\mu}$ is the law of $X$
with $X_{0}=\mu$). Indeed, the fact that  $N$ is finite and has continuous paths follows since $$N_{t}:=e^{-\lambda_{c}t}\langle h,X_{t}\rangle=e^{-\lambda_{c}t}\langle 1,X^h_{t}\rangle,$$ where $X^h$ is the $(L_0^h,\lambda_c,\alpha h;D)$-superdiffusion (see Lemma \ref{LH1} and the comment following it) with continuous total mass process.
Moreover,
\begin{align*}
&E_{\mu}\left(N_t\mid \mathcal{F}_s\right)=E_{\mu}\left(N_t\mid X_s\right)=E_{X_{s}}N_t=E_{X_{s}}\langle H(\cdot,t),X_{t}\rangle=\\
& \int_{D}E_{\delta_{x}}\langle H(\cdot,t),X_{t-s}\rangle\: X_{s}(\mathrm{d}s)=
\int_{D}\left(T_{t-s}^{L+\beta}H(\cdot,t)\right)(x)\, X_{s}(\mathrm{d}x)=\\
&\le\int_{D}H(x,s)\, X_{s}(\mathrm{d}x)=\langle H(\cdot,s),X_{s}\rangle=N_s.\end{align*}

The above analysis also shows that if $(L+\beta-\lambda_{c})^{h}$
is conservative, that is, if  $T^{(L+\beta-\lambda_{c})^{h}}1=1$, then $N$ is a continuous $P_{\mu}-$martingale, as the inequality in the previous displayed formula becomes an equality.

The continuous non-negative supermartingale $N_{t}$ has an almost sure  limit $N_\infty$ as $t\to\infty$.
Note also that   $N_t=\langle H(\cdot,t),X_{t}\rangle=e^{-\lambda_{c}t}\langle1,hX_{t}\rangle$ and $h>0$ is $C^2$ on $D$.
It follows that 
the local growth is $\mathcal{O}\left(e^{\lambda_{c}t}\right)$;
 that is, for every $B\Subset D$, 
\[
X_{t}(B)=\mathcal{O}\left(e^{\lambda_{c}t}\right)\ a.s.
\]

 \medskip
\textbf{\noun{(ii) Assume $\lambda_{c}=\infty.$ }}

In this case,   there is no $C^2$-function  $h>0$ such that $(L+\beta-\lambda)h\le0$ 
for some $\lambda\in\mathbb{R}$; see again   \cite[Chapter 4]{Pinskybook}.
Can we still get an a.s. upper estimate for the local growth?

Assume that for some  smooth positive space-time function 
$F$, inequality
(\ref{eq:supermart.prop}) holds with $F$ in place of $H$ there; 
 that is, denoting \[
F(\cdot,t)=:f^{(-t)}(\cdot),\]
we make the following assumption.

\medskip

\noindent\textbf{\noun{Assumption A}}
 \textbf{:} There exists a family $\{f^{(-t)}; \, t\ge0\}$
of smooth nonnegative functions, satisfying \[
T_{t}^{L+\beta}f^{(-t-s)}\le f^{(-s)}.\]
By smoothness we mean that  $f^{(-t)}$ is a continuous spatial function for  $t\ge 0$ and $t\mapsto f^{(-t)}(x)$ is continuous, uniformly on bounded spatial domains, at any $t_0\ge 0$.
\begin{rem}
Note that, when $\lambda_{c}<\infty$, Assumption A holds with $f^{(-t)}(\cdot):=e^{-\lambda_{c}t}h(\cdot)$, where $h$ is as before.$\hfill\diamond$
\end{rem}

As we have seen, Assumption A implies the important property that
$N_t:=\langle f^{( - t)}, X_t\rangle \ge 0$ is a $\P_{\mu}$-supermartingale. 
In order to conclude that it has an almost sure limit, we make a short detour and investigate the continuity of this supermartingale.

\begin{lem}\label{simple.lemma}
Let $\{\mu_{t},\ t\ge0\}$ be a family in $\mathcal{M}_{f}(D)$ satisfying
that $t\mapsto|\mu_{t}|$ is locally bounded, and assume that $t_{0}>0$
and $\mu_{t}\stackrel{v}{\Rightarrow}\mu_{t_{0}}$ as $t\to t_{0}$.
Assume furthermore that \[
C=C_{t_{0},\varepsilon}:=\mathsf{closure}\left(\bigcup_{t=t_{0}-\varepsilon}^{t_{0}+\varepsilon}\mathrm{supp}(\mu_{t})\right)\Subset D
\]
with some $\varepsilon>0$. Let $H:D\times\mathbb{R}_{+}\to\mathbb{R}$
be a function continuous in $x\in D$ and continuous in
time at $t_0$, uniformly on bounded spatial domains. Then $\lim_{t\to t_{0}}\langle H(\cdot,t),\mu_{t}\rangle=\langle H(\cdot,t_{0}),\mu_{t_{0}}\rangle$.\end{lem}
\begin{proof}
Using Urysohn's Lemma, there exists a continuous function $g:D\to\mathbb{R}$
such that $g(\cdot)=H(\cdot,t_{0})$ on $C$ and $g=0$ on $D\setminus D_1,$
where $C\Subset D_{1}\Subset D$. Then, \[
\lim_{t\to t_{0}}\langle H(\cdot,t_{0}),\mu_{t}\rangle=\lim_{t\to t_{0}}\langle g,\mu_{t}\rangle=\langle g,\mu_{t_{0}}\rangle=\langle H(\cdot,t_{0}),\mu_{t_{0}}\rangle,\]
 since $g\in C_c(D)$. Also, by the assumptions
on $\mu$ and $H$, for $t\in(t_{0}-\varepsilon,t_{0}+\varepsilon),$ one has
$$|\langle H(\cdot,t_{0})-H(\cdot,t),\mu_{t}\rangle|\le\sup_{x\in C}|H(x,t)-H(x,t_{0})|\,\sup_{t\in(t_{0}-\varepsilon,t_{0}+\varepsilon)}|\mu_{t}|,$$
which tends to zero as $t\to t_{0}.$
\end{proof}
\medskip{}

Recall that $\beta$ is locally bounded and the branching is quadratic. We now need a path regularity result for superprocesses. 

\begin{claim}[Continuity of $X$] Let $\mu\in\mathcal{M}_c(D)$.
 If the compact support property holds, then $(X,P_{\mu})$ has an $\mathcal{M}_f(D)$-valued,   continuous version.
(Here  continuity is meant in the weak topology of measures.)
\end{claim}
{\bf Note:} In the sequel, we will work with a weakly continuous version of the superprocess whenever the compact support property holds.

\begin{proof} 
Recall the definition of $\Omega_1$ from (\ref{CSP.all.s}); by the compact support property, we can in fact work on $\Omega_1$ instead of $\Omega$.
Pick a sequence of domains $\{D_n\}_{n\ge 1}$ satisfying that $D_n\uparrow D$ and $D_n\Subset D$ for all $n\in \mathbb N$. Define
$$\tau_n:=\inf \{t\ge 0\mid X_t(D_n^c)>0\},$$
and let  $\mathcal{F}_{\tau_{n}}$ denote the $\sigma$-algebra up to $\tau_n$, that is, $$ \mathcal{F}_{\tau_{n}}:=\{A\subset\Omega_1\mid A\cap\{\tau_n\le t\}\in \mathcal{F}_t, \forall t\ge 0\}.$$
Let $X^{D_{n}}_t$ denote the exit measure from $D_n\times [0,t)$, which is a (random) measure on  $(\partial D_n\times(0,t))\cup (D_n\times\{t\})$. Since the coefficients are locally bounded, for any fixed $n\ge 1$, $t\to X^{D_{n}}_t$ has an $\mathcal{M}_f(D))$-valued, weakly continuous version $t\to \widehat X^{D_{n}}_t$.
If $P^{(n)}$ denotes their common distribution, then
\begin{equation}\label{localiz}
P|_{\mathcal{F}_{\tau_{n}}}=P^{(n)}|_{\mathcal{F}_{\tau_{n}}}.
 \end{equation}
Let $\Omega^*:=C([0,\infty),\mathcal{M}_f(D)$ be the space of weakly continuous functions from $[0,\infty)$ to $\mathcal{M}_f(D)$ and let $\mathcal{F}^*$ denote the Borels of $\Omega^*$.
By the definition of $\Omega_1$,
\begin{equation}
\label{passes.t}
\lim_{n\to\infty}\tau_n (\omega)=\infty,\ \forall\omega\in\Omega_1,
\end{equation} and thus, it is standard to show that  the measures-valued processes
$$\left\{\widehat X_t^{D_{n}}, t\in [0,\tau_n)\right\}_{n\ge 1}$$ with distributions $(P^{(n)},\Omega^*,{\mathcal{F}_{\tau_{n}}}),n\ge 1$ have an  extension to a  process
$(X_t^*,\ t\in[0,\infty))$ with distribution $(P^*,\Omega^*,\mathcal{F}^*)$. Since $P^*$ is uniquely determined on the Borels of $\mathcal{M}_f(D)^{[0,\infty)}$ by the distributions
$(P^{(n)},\Omega^*,{\mathcal{F}_{\tau_{n}}}),n\ge 1$, therefore (\ref{localiz}) implies that $P^*=P$ on the Borels of $\mathcal{M}_f(D)^{[0,\infty)}$. Hence $X^*$ is a weakly continuous version of $X$.
\end{proof}
 
Now it is easy to see that the supermartingale $N_t$ has a continuous version: 
let us define a version of $N$ using a weakly continuous version of $X$. By Assumption A, and letting $\mu_{t}=X_{t}(\omega)$,  Lemma \ref{simple.lemma} implies the continuity of $N(\omega,t)$ at $\omega\in\Omega_1,t_{0}>0$. 
Then, since $N$ is a continuous nonnegative supermartingale, we conclude that it has an almost sure limit.

In summary, we have obtained

\begin{lem}\label{f.estimates}
Under Assumption A and assuming the compact support property (or just the existence of finite measure-valued continuous trajectories), one has 
\begin{equation}\label{estimate.with.f}
X_{t}(B)=\mathcal{O}\left(\sup_{x\in B}\frac{1}{f^{(-t)}(x)}\right)  \quad \text{a.s.}
\end{equation}
\end{lem}
In particular, the martingale property would follow if we knew that for an appropriate $f\in C^{+}(D)$, the semiorbit $t\mapsto T_{t}^{L+\beta}(f)$
can be extended from $[0,\infty)$ to $(-\infty,\infty)$. Indeed, we could then define \[
f^{(-t)}(x)=H(\cdot,t):=T_{-t}^{L+\beta}(f)(\cdot),\] which implies the statement in Assumption A with equality.
Hence, in this case, the local growth
can be upper estimated as follows. Let $B\Subset D$ be nonempty and open. Then
\begin{equation}\label{nice.little.estimate}
N_{t}=\langle H(\cdot,t),X_{t}\rangle\ge\langle H(\cdot,t)\textbf{1}_{B},X_{t}\rangle\ge\inf_{x\in B}H(x,t)\, X_{t}(B).
\end{equation}
Since $N_{t}$ has an almost sure limit, therefore 
\[
X_{t}(B)=\mathcal{O}\left(\sup_{x\in B}\frac{1}{H(x,t)}\right)  =\mathcal{O}\left(\sup_{x\in B}\frac{1}{T_{-t}^{L+\beta}(f)(x)}\right)\quad \text{a.s.}
\]

\medskip

\begin{rem}
It is of independent interest, that, using (\ref{nice.little.estimate})
one can always upper estimate the semigroup as \[
(T_{t}\textbf{1}_{B})(x)=E_x X_t(B)\le\sup_{y\in B}H^{-1}(y,t)\cdot (T_{t}(H(\cdot,t)))(x)=\sup_{y\in B}H^{-1}(y,t)
\cdot f(x),\]
 where $H$ is as before. $\hfill\diamond$ \end{rem}

\subsection{The `$p$-generalized principal eigenvalue' and a sufficient condition}
The discussion in the previous subsection gives rise to the following questions:
\begin{enumerate}
\item When is Assumption A satisfied?\textbf{ }
\item When can the semiorbit $t\mapsto T_{t}^{L+\beta}(f)$ be extended?
\end{enumerate}
We will focus  on the first question. For simplicity, use the shorthand
$T_{t}:=T_{t}^{L+\beta}$. Assume that $\vartheta$ is a 
continuous non-decreasing
 function on $[0,\infty)$, satisfying
$\vartheta (0)=0$,
\begin{equation}
\vartheta(s+t)\le C[\vartheta(s)+\vartheta(t)],\ s,t\ge0,\label{eq:bound with constant}\end{equation}
with some $C>1$ (depending on $\vartheta)$ and that $\gamma:=e^{-\vartheta}$
satisfies for all $g\in C_{c}^{+}$ that
\begin{equation}\label{general condition}
 I_g(B):=\int_{0}^{\infty}\gamma(s) \| 1_B  T_{s}g\|_\infty \,\text{d}s<\infty  
\end{equation}
for every $B\Subset D$.  
Then Assumption A is satisfied as well, since, using the monotonicity of $\gamma$, (\ref{general condition})
and dominated convergence, the family \[
\mathcal{G}_g:=\left\{ f^{(-t)}:=\int_{0}^{\infty}\gamma(s+t)T_{s}g\,\text{d}s;\ t\ge0\right\} \]
is    continuous  in $t$, uniformly on bounded spatial domains, and a trivial computation shows that
$T_{t}f^{(-t-s)}\leq f^{(-s)}$.  Assume now that the compact support property holds. By \eqref{estimate.with.f}, for a nonempty open $B\Subset D,$ \[
X_{t}(B)=\mathcal{O}\left(\left[\inf_{x\in B}\int_{0}^{\infty}\gamma(s+t)(T_{s}g)(x)\,\text{d}s\right]^{-1}\right)\ \text{a.s.},\]
and so by (\ref{eq:bound with constant}), and by the fact that $C>1$,
\begin{align}\label{estimate.with.theta}
X_{t}(B)&=\mathcal{O}\left(\gamma(t)^{-C}\left[\inf_{x\in B}\int_{0}^{\infty}\gamma(s)^{C}(T_{s}g)(x)\,\text{d}s\right]^{-1}\right)\nonumber\\
&=\mathcal{O}\left(\gamma(t)^{-C}\right)=\mathcal{O}(e^{C\vartheta(t)})\ \text{a.s.}
\end{align}
Consider now the particular case when $\vartheta(t):=\lambda t^{p}$
with $\lambda>0,\ p\geq1$ and assume that condition \eqref{general condition}
holds: there exists a non-trivial $g\geq0$ so that
\begin{equation}
f^{(0)}(B):=I_g(B)=\int_{0}^{\infty}e^{-\lambda s^{p}} \| 1_B T_{s} g \|_\infty \,\text{d}s<\infty
 \hbox{ for every  } B\Subset D.\label{eq:condition_with_p}\end{equation}
Then, by convexity, $C=C_{p}=2^{p-1}$ satisfies (\ref{eq:bound with constant}),
and so, using (\ref{estimate.with.theta}), one has
\begin{equation}\label{from.conv}
X_{t}(B)=\mathcal{O}\left(\exp(2^{p-1}\lambda t^{p})\right)\ P_{\mu}-\text{a.s.}
\end{equation}
If (\ref{eq:condition_with_p}) holds with some $\lambda>0,\ p\geq1$
and a non-trivial $g\geq0$, then we will say that the `$p$-\textit{generalized
principal eigenvalue}' of $L+\beta$, denoted by  
$\lambda_{c}^{(p)}$,  is finite and $\lambda_{p} \leq \lambda $.
More formally, we make the following definition.
\begin{defn}[$p$-generalized principal e.v.]\label{def.pgpe}
For a given $p\ge1$ we define the $p$-generalized principal eigenvalue
of $L+\beta$ on $D$ by 
\begin{eqnarray*}
\lambda_{c}^{(p)}  &:=&
 \inf\left\{ \lambda\in\mathbb{R}:  \ \exists\ \mathbf{0}\neq g\in B^+_b(D) \hbox{ so that }   \right.\\ 
 &&  \left. \hskip 0.8truein  \int_{0}^{\infty}e^{-\lambda s^{p}} \| 1_B T_{s} g \|_\infty \,\text{d}s<\infty
 \hbox{ for every  } B\Subset D\right\} .
\end{eqnarray*}
\end{defn}
For more on the $p$-generalized principal eigenvalue, see the Appendix.

\medskip

Let us now reformulate \eqref{from.conv} in terms of the $p$-generalized principal eigenvalue.
\begin{thm}[Local growth with pgpe]\label{growth.rate.with.pgpe}
Assume the compact support property and that 
$\lambda_c^{(p)}<\infty$
 with some $p\ge 1$. Then, for $B\Subset D, \varepsilon>0$, and $\mu\in\mathcal{M}_c(D)$, one has, as $t\to\infty$, that
 $$
  X_{t}(B)=\mathcal{O}\left(\exp\left( (2^{p-1}\lambda^{(p)}_c+\varepsilon)t^{p}\right) \right)\ P_{\mu}-\text{a.s.}
 $$ 
\end{thm}

\begin{rem} \label{aboutCSP}The assumption that the compact support property holds is technical in nature. We only need it to guarantee the continuity of $N$. In fact, we suspect that this assumption can be dropped in Theorem \ref{growth.rate.with.pgpe}.$\hfill\diamond$
\end{rem}
We now revisit a previous example.
\begin{example}
[The $(\frac{1}{2}\Delta,|x|,\alpha,\mathbb{R}^{d})$-superprocess]\label{revisit}
Let $D=\mathbb{R}^{d},\ L=\frac{1}{2}\Delta$, $\beta(x)=|x|$, and
$\alpha >0$, and note that the compact support property holds for this example. Although by Lemma \ref{slight.unbddness}, $\lambda_{c}=\infty$,
using (\ref{eq:two_sided_estimate}), and the estimates preceding it, it follows that $\lambda_{c}^{(3+\varepsilon)}\le 0$
for all $\varepsilon>0.$ Also, (\ref{general condition}) is satisfied
with any $\vartheta(t)=-t^{3}/2-f(t)$ and $\alpha >0,$ provided
$e^{-f(t)}$ is integrable. Let $K>0$ and $\widehat{C}:=\max\{4,K\}$.
Using the inequality $(t+s)^{3}\le4(t^{3}+s^{3})$, one obtains the
estimate
$$
X_{t}(B)=\mathcal{O}\left(\exp\left[\widehat{C}\left(t^{3}/2+f(t)\right)\right]\right),\ P_{\mu}-\text{a.s.},$$ for $\mu\in\mathcal{M}_c(D), B\Subset \mathbb R^d$ and
for any function $f\ge0$ satisfying
\[
f(t+s)\le K(f(t)+f(s)).
\]
For example, taking $f(t):=\varepsilon t^{r},\ \varepsilon>0,\ 0<r<1,$ one
obtains that for $B\Subset \mathbb R^d$,
\[
X_{t}(B)=\mathcal{O}(\exp[2t^{3}+\varepsilon't^{r}]),\ P_{\mu}-\text{a.s.}\]
\end{example}
We conclude with an open problem.
\begin{problem}
In Example \ref{revisit}, what is the exact  growth order of  $X_{t}(B),\ B\Subset\mathbb{R}^{d}$? Note that Theorem \ref{T:1.1} answers this question for the {\it global} mass when $\beta=\alpha$. See also Corollary \ref{nice.upper.est}.
\end{problem}

\section{Poissonization method  for growth rate and  for spatial spread estimates}\label{S:6}

In this section we will study the superdiffusion corresponding to the operator $\frac12 \Delta u+\beta u-\alpha u^2$ on $\R$ with $\beta (x)=|x|^p$
for $p\in (0, 2]$, and study the precise growth rate for its total mass by using a method of Poissonization.
 An upper bound for the spatial spread when $\beta (x)=|x|^2$ will also be given.
 
\subsection{General remarks on Poissonization}

We start with a Poissonization method due to Fleischmann and Swart \cite{FleischmannSwart}.
Let $(X,P)$ be the superprocess corresponding to the operator $Lu+\beta u-\alpha u^2$ on $D\subset \R^d$ and $(Z,\mathsf{P})$  the branching diffusion on $D$ with branching rate $\beta$.

 The more elementary version of Poissonization   is the fact that for a given $t>0$,
the following two spatial point processes are equal in law:
\begin{enumerate}
\item[(a)] the spatial point process $Z_t$  under $\mathsf{P}_x$;
\item[(b)] a spatial Poisson point process (PPP) $Z^*_t$ with the random intensity measure $X_t$, where $X_t$ is  the superprocess at time $t$ under $P_{\delta_{x}}$. 
 \end{enumerate}
(See Lemma 1 and Remark 2 in \cite{FleischmannSwart}.)

This is not a process level coupling, as it only matches the one-dimensional distributions. However, Fleischmann and Swart provided a  coupling of $X$ and $Z$ {\it as processes} too in \cite{FleischmannSwart}.

{\bf Convention:} Let us now introduce the following notation for convenience: when we write $\mathsf{P}_0$, it denotes the law of the process, starting with measure $\delta_0$, in case of $X$, and the law starting with a {\sf Poisson(1)} number of particles at the origin, in case of $Z$. In particular, $Z$ is the `empty process' with $\mathsf{P}_0$-probability $1/e$. ($\mathsf{E}_0$ is meant similarly.)

\medskip
Fleischmann and Swart proved that the two processes can be coupled (i.e., can be defined on the same probability space) in such a way that (with the same $\mathsf{P}_0$ because of the coupling)
\begin{equation}\label{FScoupling}
 \mathsf{P}_0[Z_t\in\cdot \mid (X_s)_{0\le s\le t}]= \mathsf{P}_0[\mathsf{Pois}(X_t)\in \cdot\mid X_t],\ \text{a.s.}\ \forall\ t\ge 0,
\end{equation}
where $\mathsf{Pois}(\mu)$ denotes the PPP with intensity $\mu$ for a finite measure $\mu$.
(See their formula (1.2) and note that in our case, the function $h$ appearing in the formula is identically one.)
Formula \eqref{FScoupling} says that the conditional law of $Z_t$, given the history of $X$ up to $t$, is the law of a PPP with
intensity $X_t$. (In fact they prove an even stronger version, involving {\it historical processes} in their Theorem 6.)
\begin{assumption}
Because of the Poissonization method, we will assume that $\alpha=\beta$ (See Lemma 1 and Remark 2 in \cite{FleischmannSwart})
\end{assumption}
\noindent{\bf Note:} for $\alpha\ge \beta$, it is easy to see that the upper bounds still hold. (Reason: we have an extra `death' term if they are not equal. See again \cite{FleischmannSwart}.)

\medskip
We will use the  
abbreviation 
  FALT:=`for arbitrarily large times'=for some sequence of times tending to $\infty$, and
FALn:=`for arbitrarily large $n$s=for some sequence of integers tending to $\infty$'.

\subsection{Upgrading the Fleischmann-Swart coupling to stopping times}
\bigskip
We need to upgrade the coupling result to nonnegative, finite stopping times, as follows. Let $\mathcal{F}^X$ denote the canonical filtration of $X$, that is, let $\mathcal{F}^X:=\{\mathcal{F}_t^X;t\ge 0\}.$
\begin{thm}[Enhanced coupling]\label{enhanced}
Given the Fleischmann-Swart coupling, it also holds that for an almost surely finite and nonnegative $\mathcal{F}^{X}$-stopping time $T$,
$$
{\mathsf P_0}\left[Z_T\in\cdot\mid (X_s)_{0\le s\le T}\right]={\mathsf P_0}\left[
\mathsf {Pois}(X_T)\in\cdot \mid X_T\right],\ \text{a.s.}
$$
\end{thm}
\begin{rem} Note that
\begin{enumerate}
\item The lefthand side is just another notation for ${\mathsf P_0}\left[Z_T\in\cdot\mid \mathcal{F}^X_{T}\right].$ 
Actually, as the proof below reveals,  a bit stronger result is also true: $\mathcal{F}^X_{T}$ can be replaced even by $\mathcal{F}^X_{T^{+}}$.
\item For the time of extinction of $X$,  the result is not applicable. Indeed, using that $\alpha=\beta$, it is easy to show that for this $T$, we have $T=\infty$ with positive probability.
\end{enumerate}
\end{rem}
\begin{proof}
As usual, we will  approximate $T$ with a decreasing sequence of countable range stopping times.

We need the facts that, as measure-valued processes, both $X$ and $Z$  are right-continuous, and $X$ is in fact continuous. We proved weak continuity  for $X$, see Claims 24 and 26. For $Z$, right-continuity is elementary.

We now turn to the proof of the statement of the theorem. 
Following pp. 56--58 in \cite{ChungWalsh}, take a general nonnegative $\mathcal{F}^{X}$-stopping time $T$, and let 
 $$ {\mathbb T}:=\{k/2^m\mid k,m\ge 0\}  
$$
be the dyadic set. 
For $n\ge 1$, define the 
 ${\mathbb T}$-valued
 $\mathcal{F}^{X}$-stopping time (in \cite{ChungWalsh}, `strictly optional' is used instead of `stopping') 
$$T_n:=\frac{\lfloor2^{n}T\rfloor+1}{2^{n}}.$$
Then $T_n\downarrow T$ uniformly in $\omega$. In fact (see \cite{ChungWalsh}), 
\begin{equation}\label{wedge}
\mathcal{F}^X_{T^{+}}=\bigwedge_{n=1}^{\infty}\mathcal{F}^X_{T_{n}},
\end{equation}
where the righthand side is the intersection of the $\sigma$-algebras.

Fix $n\ge 1$. Since $T_n$ has countable range, 
$$
{\mathsf P_0}\left[Z_{T_{n}}\in\cdot\mid (X_s)_{0\le s\le T_{n}}\right]={\mathsf P_0}\left[
\mathsf {Pois}(X_{T_{n}})\in\cdot \mid X_{T_{n}}\right],\ \text{a.s.}
$$
Indeed, using Laplace-transforms and the Campbell formula for PPP, this is equivalent to the assertion that for each bounded and continuous $f\ge 0$,
\begin{equation}\label{Campbell}
{\mathsf E_0}\left[\exp \langle -f,Z_{T_{n}}\rangle \mid (X_s)_{0\le s\le T_{n}}\right] =
\exp\left(-\int_{\mathbb R^{d}}(1-e^{-f(x)})X_{T_{n}} (\mathrm{d}x)\right)\hbox{ a.s.}
\end{equation}
To provide a rigorous proof for \eqref{Campbell}, let $A\in \mathcal{F}^X_{T_{n}}$ and for 
 $t\in {\mathbb T}$, define
 $$
A_t:=A\cap \{T_n=t\}\in\mathcal{F}^X_t.
$$
Since $T_n$ has countable range, we have almost surely,
$$
{\mathsf E_0}\left[\exp(\langle -f,Z_{T_{n}}\rangle) ; A\right]=
 \sum_{t\in {\mathbb T}}{\mathsf E_0}\left[\exp(\langle Z_{t},-f\rangle) ; A_t\right].
$$
Since $A_t\in\mathcal{F}^X_t$,  
by the Fleischmann-Swart coupling, the last sum equals 
$$
\sum_{t\in {\mathbb T}} \exp\left[-\int_{\mathbb R^{d}}(1-e^{-f(x)})X_{t} (\mathrm{d}x)\right]{\mathsf P_0}(A_t), \hbox{ a.s.},
$$
which is the same as 
$${\mathsf E_0} \left[\exp\left[-\int_{\mathbb R^{d}}(1-e^{-f(x)})X_{T_{n}} (\mathrm{d}x)\right];A\right], \hbox{ a.s.}$$
This completes the proof of \eqref{Campbell}.

Now let $n\to \infty$.
By the continuity of $X$, the a.s. limit of the righthand side in \eqref{Campbell} is 
$$\exp\left[-\int_{\mathbb R^{d}}(1-e^{-f(x)})X_{T} (\mathrm{d}x)\right].$$
Thus, it remains to show that  a.s.,
$$\lim_{n\to\infty}{\mathsf E_0}\left[\exp\langle -f,Z_{T_{n}}\rangle \mid \mathcal{F}_{T_{n}}^X\right]={\mathsf E_0}\left[\exp\langle -f,Z_{T}\rangle\mid \mathcal{F}_T^X\right].$$
Note that we already know that the a.s. limit exists and just have to identify it. Hence, it is enough to prove that ${\mathsf E_0}\left[\exp\langle -f,Z_{T}\rangle\mid \mathcal{F}_T^X\right]$ is the limit in $L^1$, for example.

Clearly,
\begin{eqnarray*}
&&{\mathsf E_0}\left[\exp\langle -f,Z_{T_{n}}\rangle \mid \mathcal{F}_{T_{n}}^X\right] \\
&=&   {\mathsf E_0}\left[\exp\langle -f,Z_{T}\rangle \mid \mathcal{F}_{T_{n}}^X\right]+   
{\mathsf E_0}\left[\exp\langle -f,Z_{T_{n}}\rangle-\exp\langle -f,Z_{T}\rangle \mid \mathcal{F}_{T_{n}}^X\right] \\
&=:& A_n+B_n.
\end{eqnarray*}
Then $\lim_{n\to \infty}B_n=0$  in $L^1$, because
\begin{eqnarray*}
&&{\mathsf E_0}\left(|{\mathsf E_0}\left[\exp\langle -f,Z_{T_{n}}\rangle-\exp\langle -f,Z_{T}\rangle \mid \mathcal{F}_{T_{n}}^X\right]|\right)\\
&\le & 
{\mathsf E_0}\left({\mathsf E_0}\left[|\exp\langle -f,Z_{T_{n}}\rangle-\exp\langle -f,Z_{T}\rangle| \mid \mathcal{F}_{T_{n}}^X\right]\right) \\
&=&    {\mathsf E_0}\left(|\exp\langle -f,Z_{T_{n}}\rangle-\exp\langle -f,Z_{T}\rangle|\right)\to 0  \quad \hbox{as }  n\to\infty,
\end{eqnarray*}
where the last step uses bounded convergence along with the $\omega$-wise right continuity of $Z$.

 Finally, since $T_n$ is decreasing,
$$\lim_{n\to\infty}A_n={\mathsf E_0}\left[\exp\langle -f,Z_{T}\rangle\mid \mathcal{F}^{X}_{T^{+}}\right], \text{a.s. and in}\ L^1$$
by \eqref{wedge} and the (reverse) Martingale Convergence Theorem for conditional expectations.
\end{proof}

\subsection{The growth of the total mass; proof of Theorems \ref{GCT} and \ref{T:1.1}}

The almost sure growth rate of the total mass has been described in \cite{BBHH2010} for $Z$ on $\R$ with  $\beta(x)=Cx^2, C>0$, and in \cite{BBHHM2015} for the case when $\beta(x)=|x|^p, 0\le p<2$.
For the first  case, the authors have verified double-exponential growth:
$$\lim_{t\to\infty}(\log\log |Z_t|)/t=2\sqrt{2}C,\ \text{a.s.}$$
For $\beta(x)=|x|^p, 0\le p<2$, it has been  shown that
$$\lim_{t\to\infty}\frac{1}{t^{\frac{2+p}{2-p}}}\log |Z_t|=K_p,\ \text{a.s.},$$
where $K_p$ is a positive constant, determined by a variational problem. (Also, for $p\in (0,2]$, right-most particle speeds are given.) Note that these proofs carry through for the case when $Cx^2$ (resp. $|x|^p$) is replaced by $1+Cx^2$ (resp. $1+|x|^p$), too.

We are going to utilize these results, as well as a general comparison result which produces an upper/lower bound on $|X|$ once one has an upper/lower bound on $Z$. This comparison result is based on Poissonization.

But first we need some basic facts about general superdiffusions. In what follows, we are going to use several results from \cite{AOP99}. Although in that paper the assumption $\lambda_c<\infty$ was in force, the results are still applicable in our setting. The reason is that for all the results we are using in the $\lambda_c =\infty$ case,  the proof only uses the local properties of the coefficients.  

Recall that  $X$ satisfies the compact support property in a number of interesting cases. (See  Claims \ref{csp.holds} and \ref{csp.still.holds}.)

If $S$ stands for survival, then $P_{\delta_{x}}(S^c)=e^{-w_{\mathrm{ext}}(x)}$, where  $w_{\mathrm{ext}}$ is a particular nonnegative solution to the steady state equation
    \begin{equation}\label{steady.state}
        L u+\beta u-\alpha u^2=0. 
    \end{equation} (See Theorem 3.1 in \cite{AOP99}.) Its finiteness and the fact that it solves the equation, follows the same way as in \cite{AOP99}. Finiteness follows from Lemma 7.1 in \cite{AOP99}  (It says that given any $t,R>0$, 
         with positive probability 
    the process may die out by $t$ without ever charging  a ball of radius $R$ around $x$. All one needs is that locally, $\beta$ ($\alpha$) is bounded from above (bounded away from zero).) 
    
By Theorem 3.3 in \cite{AOP99}, $w_{ext}=w_{max}$, whenever the compact support property holds, where $w_{\mathrm{max}}$ denotes the maximal nonnegative solution to the steady state  equation \eqref{steady.state}.

Write simply $w$ for $w_{\mathrm{ext}}$.  Assuming that $D=\R^d,\alpha=\beta$ and that the coefficients of $\frac{1}{\alpha} L$ are  bounded from above\footnote{Actually certain growth can be allowed.} (for example $L=\Delta/2$ and $\alpha$ is bounded away from zero), we are going to show that $w\le 1$.
Clearly, $w\le w_{\mathrm{max}}$ and  $w_{\mathrm{max}}$ is also the maximal nonnegative solution to  $\frac{1}{\alpha}L u+u-u^2=0$, corresponding to the $(\frac{1}{\alpha} L, 1, 1;\R^d)$-superprocess, denoted by $\widehat{X}$. 

Denote by $\widehat{w}_{\mathrm{ext}}$ and $\widehat{w}_{\mathrm{max}}$ the corresponding functions for $\widehat{X}$. 
When $1/\alpha \cdot L$ has coefficients bounded from above,  the compact support property holds for $\widehat{X}$ (see Theorem 3.5 in \cite{AOP99}); therefore $$\widehat{w}_{\mathrm{ext}}=\widehat{w}_{\mathrm{max}}={w}_{\mathrm{max}}\ge w.$$ 
Thus, $w\le 1$ follows from $\widehat{w}_{\mathrm{ext}}\le 1$,
which in turn follows from Proposition 3.1 in \cite{AOP99}.  

Next, we need some Poissonian estimates.
\begin{lemma}[Poissonian tail estimates]\label{Poiss.Lemma}
If $Y$ is a Poisson random variable with parameter $\lambda$, then 
\begin{align*}P(Y\le y)\le e^{y-\lambda}\left(\frac{\lambda}{y}\right)^y,\ \text{for}\ y<\lambda;\\
P(Y\ge y)\le e^{y-\lambda}\left(\frac{\lambda}{y}\right)^y,\ \text{for}\ y>\lambda.
\end{align*}
In particular, 
for $k<1$ we have $P(Y\le k\lambda)\le C_k^{\lambda},$ and 
for $k>1$ we have $P(Y\ge k\lambda)\le C_k^{\lambda},$
where $$C_k:=(e/k)^k\cdot (1/e)<1.$$
\end{lemma} 
\begin{proof}
Use the Chernoff-bound for the first part. The statement that $C_k<1$, after taking logarithm and defining $z=\ln k$, becomes
$1-z<e^{-z}$.
\end{proof}

\noindent{\it Proof of Theorem \ref{GCT}}. 
We will utilize Lemma \ref{Poiss.Lemma} and Theorem \ref{enhanced}. 

Keeping the Poissonization method in mind, let both $Z$ and $X$ be defined on the probability space $(\Omega,\mathsf{P})$. As before, we will write $\mathsf{P}_0$ to indicate that $Z$ and $X$ are started with a {\sf Poisson(1)}-number of particles at zero and with $\delta_0$, respectively.

\medskip

(i)    For $\varepsilon,t>0$, define the events
\begin{eqnarray*}
E^t_{\varepsilon} & := & \{|X_t|> (1+\varepsilon)f(t)\}; \\
F^t_{\varepsilon/2}& := & \left\{\frac{|Z_t|}{f(t)}\le 1+\varepsilon/2\right\}, \\
G_{\varepsilon/2}^{t}& := & \left\{\frac{|Z_t|}{f(t)}> 1+\varepsilon/2\right\}=\left(F_{\varepsilon/2}^{t}\right)^c.
\end{eqnarray*}
Define also
\begin{eqnarray*}
E_{\varepsilon} & := & \{|X_t|> (1+\varepsilon)f(t),\ \text{FALT}\}; \\
H_{\varepsilon/2}& := & \left\{\frac{|Z_t|}{f(t)}> 1+\varepsilon/2,\ \text{FALT}\right\}. 
\end{eqnarray*}
Since
 $$ \mathsf{P}_0\left(\limsup_t \frac{|X_t|}{f(t)}> 1\right)\le \sum_{m\ge 1} 
 \mathsf{P}_0\left( E_{\frac1m}\right),$$
it is enough to show that for $\varepsilon>0$,
$ \mathsf{P}_0\left(E_{\varepsilon}\right)=0.$

Fix $\varepsilon>0$.
For  $\omega\in E_{\varepsilon}$, define a sequence of random times $(t_n)_{n\ge 0}=(t_n(\omega))_{n\ge 0}$ recursively, by  $t_0:=0$ and
$$t_{n+1}:=
\inf\{
t>t_n \mid |X_t |> (1+\varepsilon)f(t)\ \text{and}\ f(t)\ge n+1)
\},\ n\ge 0.$$ (For convenience, define $t_n(\omega)$ for $\omega\in \Omega\setminus E_{\varepsilon}$ in an arbitrary way.) Recall that we have proved that $X$ has weakly continuous trajectories, hence $|X|$ is continuous. Thus $t_n$ is an $\mathcal{F}^X$-stopping time; let $Q_n$ denote its distribution on $[0,\infty)$.

Clearly,  $\liminf_n G_{\varepsilon/2}^{t_{n}}\subset H_{\varepsilon/2}$. Hence, if we  show that
\begin{eqnarray}\label{kellene}
\mathsf{P}_0\left( E_{\varepsilon}\cap \left(\liminf_n G_{\varepsilon/2}^{t_{n}}\right)^c\right)&=&\nonumber \\
\mathsf{P}_0\left( E_{\varepsilon}\cap \left(\limsup_n F_{\varepsilon/2}^{t_{n}}\right)\right)&=&\mathsf{P}_0\left(\limsup_n\, (F_{\varepsilon/2}^{t_{n}}\cap E_{\varepsilon})\right)=0,
 \end{eqnarray}
 then $\mathsf{P}_0\left(E_{\varepsilon}\right)>0$ implies that
$\mathsf{P}_0\left(H_{\varepsilon/2}\right)>0$,
which  contradicts \eqref{feltetel}, and we are done.

To show \eqref{kellene}, by  Borel-Cantelli, it is sufficient to verify that
\begin{equation}\label{summab}\sum_n \mathsf{P}_0\left(F_{\varepsilon/2}^{t_{n}}\cap E_{\varepsilon}\right)<\infty.
\end{equation}
To achieve this, fix $n\ge 1$. Applying Theorem \ref{enhanced} with $T=t_n$, we have  that 
\begin{eqnarray}\label{two.cond,2ndpart} 
\mathsf{P}_0\left(F^{t_{n}}_{\varepsilon/2}\mid E_{\varepsilon}\right)&=&\mathsf{E}_0\left[\mathsf{P}_0\left(F^{t_{n}}_{\varepsilon/2}\mid (X_s)_{0\le s\le t_n}\right)\mid E_{\varepsilon}\right]\nonumber\\
&=& \mathsf{E}_0\left[\mathsf{P}_0\left( \mathsf{Pois}\left(|X_{t_{n}}|)\le(1+\varepsilon/2)f(t_{n}\right)\mid X_{t_{n}}\right)\mid E_{\varepsilon}\right].
 \end{eqnarray}
Set $k=\frac{1+\varepsilon/2}{1+\varepsilon}$.
By \eqref{two.cond,2ndpart} along with Lemma \ref{Poiss.Lemma} (recall  $C_k<1$ and that $f(t_{n})\ge n$),
it follows that, almost surely on $E_{\varepsilon}$,
$$\mathsf{P}_0\left( \mathsf{Pois}\left(|X_{t_{n}}|)\le(1+\varepsilon/2)f(t_{n}\right)\mid X_{t_{n}}\right)\le C_k^{(1+\varepsilon)n}.$$
Thus, $$\mathsf{P}_0\left(F_{\varepsilon/2}^{t_{n}}\cap E_{\varepsilon}\right)\le \mathsf{P}_0\left(F_{\varepsilon/2}^{t_{n}}\mid E_{\varepsilon}\right)\le C_k^{(1+\varepsilon)n},$$
 and since $C_k<1$, \eqref{summab} follows.
 
 \medskip
  
  (ii) 
The  proof is very similar to that of $(i)$, except that we now work on $S$, the condition $|X_t |< (1-\varepsilon)f(t)$ has to be replaced by $|X_t |>(1+\varepsilon)f(t)$ throughout, and we now define
$$t_{n+1}:=
\inf\{
t>t_n \mid n+1<|X_t |< (1-\varepsilon)f(t)
\},\ n\ge 0.$$
(In this case we set $k:=(1-\varepsilon/2)/ (1-\varepsilon)>1.$)
The summability  at the end is still satisfied because of the $n+1<|X_t |$ part in the definition.

Finally, the  statement given by the last sentence in (ii) follows from the fact that $\exp(-\langle w,X_t\rangle)$ is a martingale with expectation $e^{-w(0)}$. This, in turn, is a consequence of the Markov property and the fact that $\mathsf{P}_{\mu}(S^c)=e^{-\langle w,\mu\rangle} $. (See the beginning of the subsection for the definition of $w$.) The martingale limit's expectation cannot be less than $e^{-w(0)}$, but on extinction, the limit is clearly one, and the probability of extinction is also $e^{-w(0)}$. Hence the limit must be zero on $S$, that is $\langle w,X_t\rangle\to\infty$. But we have already checked that $w\le 1$ holds under the assumption.

Theorem \ref{T:1.1} in the Introduction is a consequence of Theorem \ref{GCT}.

\medskip

 \noindent{\it Proof of Theorem \ref{T:1.1}}. 
We treat the non-quadratic case; the quadratic case is similar.
Also, we only discuss the upper estimate; the lower estimate is similar.

Denote $h(t):=Kt^{\frac{2+p}{2-p}}$.
For the upper estimate, we need that the event 
$$E:=\left\{\limsup_t \frac{\log |X_t|}{h(t)}>1\right\}$$
is a zero event. But $E$ occurs if and only if
$$\exists \varepsilon>0:\  \frac{\log |X_t|}{h(t)}>(1+\varepsilon),\ \text{FALT}\Leftrightarrow \exists \varepsilon>0:\   |X_t|>\exp(h(t)(1+\varepsilon)),\ \text{FALT}.$$
Now
$$E\subset A:=\left\{\exists \varepsilon>0:\  \limsup_t \frac{ |X_t|}{\exp(h(t)(1+\varepsilon))}\ge 1.\right\}$$
Write $$\frac{ |X_t|}{\exp(h(t)(1+\varepsilon))}=\frac{ |X_t|}{\exp(h(t))}\frac{1} {\exp(h(t)(\varepsilon))}.$$
The $\limsup$ of the first term is almost surely bounded by one by  Theorem \ref{GCT} and by the corresponding result\footnote{The result for $Z$ is true even if $Z$ starts with $k\ge 1$ particles instead of a single one, as the process can be considered as an independent sum of $k$ processes, each starting with a single particle.} on $Z$, while the second term tends to zero. Working with countably many $\varepsilon$'s (say, $\varepsilon_m:=1/m$), we see that $A$ is a zero event indeed. 
\qed

\subsection{Upper bound for the spatial spread}

\begin{thm}[Upper bound for the spread]\label{ub.spread} Let $\varepsilon>0$. For $d=1$ and $\beta(x)=\alpha(x)=1+|x|^2$, we have
$$P _{\delta_{0}}\left( \lim_{t\to\infty} X_t(B^c(0,\exp((\sqrt{2}+\varepsilon)t)))=0\right)=1.$$
\end{thm}
\begin{proof}
Clearly, it is enough to prove that for any $\delta>0$,
\begin{equation}\label{clearly}
P_{\delta_{0}} \left( \exists T: X_t(B^c(0,\exp((\sqrt{2}+\varepsilon)t)))\le \delta,\ \text{for}\ t>T\right)=1.
\end{equation}

Harris and Harris  \cite{HH2009} have shown for the (one-dimensional) discrete branching Brownian motion $Z$ with branching rate $\beta$ that
$$\mathsf{P}_0 \left(\limsup_{t\to\infty}\frac{\log M_t}{t}\le \sqrt{2}\right)=1,$$
where $M_t$ is the rightmost particle's position. (Again, they considered $\beta(x)=|x|^2$, but the proof carries through for $\beta(x)=1+|x|^2$ as well.) By symmetry, it follows that
$$\mathsf{P}_0 \left(\limsup_{t\to\infty}\frac{\log \rho_t}{t}\le \sqrt{2}\right)=1,$$
where $\rho_t$ is the radius of the minimal interval containing $\mathrm{supp}(Z_t)$.
That is,
\begin{equation}\label{rho.small}
\mathsf{P}_0 \left( \rho_t> \exp((\sqrt{2}+\varepsilon)t),\ \text{FALT}\right)=0.
\end{equation}
Returning to \eqref{clearly}, we need to show that 
$$p_{\varepsilon}:=P \left( X_t\left(B^c(0,\exp((\sqrt{2}+\varepsilon)t))\right)>\delta,\ \text{FALT}\right)=0.$$
Indeed, suppose that $p_{\varepsilon}>0$. Recall that for a PPP, the probability that a set with mass at least $\delta$ (by the intensity measure)
is vacant is at most $\exp(-\delta)$.

As before, consider the `Poissonization coupling' of the processes $Z$ and $X$.
By the reverse Fatou inequality,\footnote{Which is $\limsup P(A_t)\le P(\limsup A_t)$.} on the event 
$$\left\{ X_t\left(B^c(0,\exp((\sqrt{2}+\varepsilon)t))\right)>\delta,\ \text{FALT}\right\},$$ the discrete point process charges $B^c(0,\exp((\sqrt{2}+\varepsilon)t))$
FALT, with probability at least $e^{-\delta}$.
It follows that the probability in \eqref{rho.small} is positive;
a contradiction.
\end{proof}
\begin{rem} It is not difficult to see that this upper estimate remains valid if $\alpha\ge\beta$ instead of $\alpha=\beta$.$\hfill\diamond$
\end{rem}

\section{Appendix: properties of $\lambda_c^{(p)}$}
 
Recall the definition of the $p$-generalized principle eigenvalue, $\lambda_c^{(p)}$  from Definition \ref{def.pgpe}.
First note that $\lambda_{c}^{(1)}=\lambda_{c},$ because (c.f. Chapter 4 in
\cite{Pinskybook})
\begin{align*}
 \lambda_{c}^{(1)}&:=\inf\left\{ \lambda\in\mathbb{R}:    G_\lambda (x,B)  \hbox{ is locally bounded in $D$ for  some }B\Subset D\right\} 
 \\
& =\lambda_{c,}
\end{align*}
where 
\[
 G_\lambda (x,B):= G^{L+\beta-\lambda}(x,B):=\int_{t=0}^{\infty}p^{L+\beta-\lambda}(t,x,B)\, \mathrm{d}t,
\]
and  $p^{L+\beta-\lambda}(t,\cdot,\cdot)$ denotes the transition
kernel for $L+\beta-\lambda$ on $D$. (When finite on compacts, the measure 
 $G_\lambda (x,\cdot)$
is called the Green measure for $L+\beta-\lambda$ on $D$.)

Next, note that if one replaces the semigroup in the definition of $\lambda_c^{(p)}$ by that of some compactly embedded ball in $D$ (with zero boundary condition), then $\lambda_c^{(p)}$ will definitely not increase, while even this modified value is different from $-\infty$, as $\beta$ is bounded on the ball. This leads to
\begin{proposition}
One has $\lambda_c^{(p)}\in (-\infty,\infty]$.
\end{proposition}
Moreover, the following  comparison principle holds:
\begin{proposition}[Comparison]
Let $p>q\ge 1$.

\begin{itemize}
\item[({\sf a})] If $ \lambda_c^{(q)} \geq 0$, then $\lambda_c^{(p)}\le \lambda_c^{(q)}$.

\item[({\sf b})] If $\lambda_c^{(q)} \leq 0$, then $\lambda_c^{(p)}\ge \lambda_c^{(q)}$.

\item[({\sf c})] If $\lambda_c^{(q)}=0$, then $\lambda_c^{(p)}= \lambda_c^{(q)}=0$. In particular, if  $\lambda_c=0$, then $\lambda_c^{(p)}=0$ for all $p>0$.
\end{itemize}
\end{proposition}

\begin{proof}   ({\sf a}): Suppose it is not true and take a $\lambda$ s.t.
$$\lambda_c^{(p)}>\lambda>\lambda_c^{(q)}\ge 0.$$
Then there is a  non-trivial non-negative $g\in C_{c}(D)$ so that for every  
 $B\Subset D,$ 
\[
 \int_{0}^{\infty}e^{-\lambda t^{q}}  \| 1_B T_{t}g \|_\infty \, \mathrm{d}t<\infty,
\]
but the same fails when $t^q$ in the integral is replaced by $t^p$; contradiction.

\bigskip

\noindent  ({\sf b}): Suppose it is not true and take a $\lambda$ s.t.
$$\lambda_c^{(p)}<\lambda<\lambda_c^{(q)}\le 0.$$
Then there is a  non-trivial non-negative $g\in C_{c}(D)$ so that for every  
 $B\Subset D,$   
\[
 \int_{0}^{\infty}e^{-\lambda t^{p}}  \| 1_B T_{t}g \|_\infty \, \mathrm{d}t<\infty,
\]
but the same fails when $t^p$ in the integral is replaced by $t^q$; contradiction.

\bigskip

\noindent   ({\sf c}): Clear from ({\sf a}) and ({\sf b}).
\end{proof}

An equivalent formulation of the definition of $\lambda_c^{(p)}$ is given in the following result.
\begin{thm} Assume that
$\beta$ is bounded from below, 
that is, $\inf_D \beta>-\infty$. Then
 \begin{align*}
\lambda_{c}^{(p)}
= \inf\left\{ \lambda\in\mathbb{R}:  
 \int_{0}^{\infty}e^{-\lambda s^{p}} \| 1_B T_{s}g \|_\infty \,\mathrm{d}s<\infty,
 \ \forall
 g\in C_{c}^{+}(D)\ \forall  B\Subset  D \right\} .
\end{align*}
\end{thm}

\begin{proof}
If suffices to show that for every $g\in C_c^+(D)$,
\begin{equation}\label{e:3.8}
 \lambda_{c}^{(p)}\geq  
\inf\left\{ \lambda\in\mathbb{R}:  
 \int_{0}^{\infty}e^{-\lambda s^{p}} \| 1_B T_{s}g \|_\infty \,\mathrm{d}s<\infty 
 \hbox{ for every }   B\Subset  D \right\} .
\end{equation}
For every $\varepsilon >0$, by the definition of $\lambda_c^{(p)}$, there is some $g_0\in B^+_b(D)$ with $g_0\neq\mathbf{0}$ such that
$\int_0^\infty e^{-(\lambda_c^{(p)}+\varepsilon)t^p}  \| 1_B T_tg_0 \|_\infty \,\mathrm{d}t <\infty$ for every $B\Subset D$.
Let $g$ be an arbitrary function in $  C_c^+(D)$. 
We denote by $f^-$ the negative part of a function $f$; 
that is, $f^-(x):= \max\{0, -f(x)\}$.
Let $\{T^{(1)}_t; t\geq 0\}$ be the semigroup for the Schr\"odinger operator
$L-\beta^-$; that is,
$$ T^{(1)}_t f(x):=\E_x \left[ \exp \left( -\int_0^t \beta^-(X_s) \,\mathrm{d}s\right) f(X_t); t< \tau_D \right].
$$
Note that for any $f\geq 0$ and $t\ge 0$, one has $T_t  f\geq T_t^{(1)}f$.
Since by the strong Feller property and irreducibility of $\{T^{(1)}_t, t\geq 0\}$,
$T^{(1)}_1g_0\in C_b(D)$ and $T_1 g_0>0$, there is a constant $c>0$ such that
$0\leq g\leq cT^{(1)}_1 g_0\le cT_1 g_0$. 
Consequently, for any $B\Subset D$, 
\begin{align*}
\int_0^\infty e^{-(\lambda_c^{(p)}+2\varepsilon)t^p}  \| 1_B T_tg \|_\infty \,\mathrm{d}t
&\leq c \int_0^\infty e^{-(\lambda_c^{(p)}+2\varepsilon )t^p} \| 1_B T_{t+1} g_0 \|_\infty \,\mathrm{d}t \\
&= c \int_1^\infty e^{-(\lambda_c^{(p)}+2\varepsilon )(t-1)^p}  \| 1_B T_t g_0 \|_\infty \,\mathrm{d}t \\
&\leq  c_1 \int_0^\infty e^{-(\lambda_c^{(p)}+ \varepsilon )t^p} \| 1_B T_t g_0 \|_\infty \,\mathrm{d}t,
\end{align*}
 which is finite a.e. on $D$. This shows that
 $$ \inf\left\{ \lambda\in \mathbb{R}: \int_{0}^{\infty}e^{-\lambda s^{p}} \| 1_B  T_{s}g \|_\infty ) \,\text{d}s<\infty \hbox{ for every  } 
 B \Subset  D \right\} \leq \lambda_c^{(p)}+2 \varepsilon.
$$
Since this holds for every $\varepsilon>0$, we conclude that  \eqref{e:3.8}
and hence the theorem holds.
\end{proof}
\subsection{An estimate of $\lambda_c^{(p)}$ in a particular case}\label{estimate.particular}
Consider the case when $D=\mathbb R^d, L=\frac{1}{2}\Delta$, $\beta(x)=|x|^{\ell}, \ell>0$ and $\alpha>0$. It is natural to ask for what $p$ do we have $0<\lambda_c^{(p)}<\infty$? Can we estimate it?

Below we give a bound for $0<\ell<2$.
By \eqref{e:5.2} and the paragraph following it, if $(B,\P)$ is a Brownian motion, then
\begin{eqnarray}\label{e:c1}
&& e^{-|x|^\ell} \E_0 \left[ \exp \left(\int_0^t 2^{-\ell}
 |x+B_s|^\ell \mathrm{d}s\right) \right]  \\
&\leq&  \E_x \left[ \exp \left(\int_0^t    |B_s|^\ell \mathrm{d}s\right) \right]
\leq
e^{2^\ell |x|^\ell }\E_0 \left[ \exp \left(\int_0^t 2^\ell 
|B_s|^\ell \mathrm{d}s\right) \right], \nonumber
\end{eqnarray}
while
$$
\E_0 \left[ \exp \left( a \int_0^t |B_s|^\ell \mathrm{d}s\right) \right]
= \int_0^\infty  a  t^{1+\ell/2} e^{a u t^{1+\ell/2} } \left(
e^{-\frac12 c_\ell u^{2/\ell} (1+o(1))} \right) \mathrm{d}u. 
$$
Hence there is a constant $c>0$ so that 
$$
T_t 1 (x) \leq e^{2^\ell |x|^\ell } 2^\ell   t^{1+\ell/2}
\int_0^\infty   e^{2^\ell u t^{1+\ell/2} } 
e^{-c u^{2/\ell}  }  \mathrm{d}u. 
$$
Let $c_1$ be the solution of 
$(2^{\ell} +1)v=cv^{2/\ell}$; that is, $c_1=((2^\ell +1 )/c)^{\ell/(2-\ell)}$. 
Note that for $v\geq c_1$, $cv^{2/\ell}\geq 2^{\ell}v + v $.
Using the shorthands $$L:=(2+\ell)/(2-\ell)\in (1,\infty),\ \ k_{\ell}(x):=e^{2^\ell |x|^\ell } 2^\ell,$$ a change of variable $u=t^\eta v$ with $\eta =L \ell/2$ yields that
\begin{eqnarray*}
T_t 1 (x) &\leq& k_{\ell}(x)\,   t^{L}
\int_0^\infty   \exp\left( t^{L} (2^\ell v-cv^{2/\ell} )\right)   \mathrm{d} v \\
&\leq&  k_{\ell}(x)\,   t^{L} \left( \int_0^{c_1}
 e^{v 2^\ell t^{L} }   \mathrm{d} v +  \int_{c_1}^\infty
 e^{- v t^{L}  }   \mathrm{d} v \right) \\
&\leq & k_{\ell}(x)  \left( \exp\left(c_1 2^\ell   
t^{L}  \right) +1\right). 
\end{eqnarray*}
Thus, for every $\gamma > e^{c_1 2^\ell }$,
$$
\int_0^\infty e^{-\gamma t^{L}  } T_t 1(x) \mathrm{d}t <\infty.
$$
It follows that, with $p=L$, $\lambda^{(p)}_c\leq \gamma$ and so 
 \begin{equation}\label{upper.est.p-GPE}
\lambda^{(p)}_c  \leq e^{c_1 2^\ell }.
 \end{equation}
The exponent $L=(2+\ell)/(2-\ell)$ is sharp when $\ell=1$, as can be seen from \eqref{BM.with.p-potential}. 
 
\medskip
From \eqref{upper.est.p-GPE} along with Theorem \ref{growth.rate.with.pgpe}, we obtain the following upper estimate. (Cf. Example \ref{revisit}.)
\begin{cor}[Upper estimate for SBM with $|x|^{\ell}$-potential when $0\le \ell<2$]
\label{nice.upper.est}
For the $\left((1/2)\Delta,|x|^{\ell},\alpha;\R^d\right)$-superdiffusion, with $0\le \ell<2$, one has that almost surely, as $t\to\infty$,
$$X_t(B)=\mathcal{O}\left(\exp\left\{\mathrm{const}\cdot t^{\frac{2+\ell}{2-\ell}}\right\}\right),\ B\Subset \mathbb R^d,$$ provided that $\alpha$ is such that the compact support property holds. 
\end{cor}
Regarding the assumption on $\alpha$, recall Remark \ref{aboutCSP}.

\bigskip
{\bf Acknowledgement}: The second author  thanks   J. Feng for bringing \cite{SharpPoincare1,SharpPoincare2} to his attention, and to J. Berestycki, S. Kuznetsov, D. Stroock and J. Swart for valuable discussions. The hospitality of Microsoft Research, Redmond and of the University of Washington, Seattle are also gratefully acknowledged by the second author.

\selectlanguage{american} 
\bibliographystyle{amsalpha}
\bibliography{Jancsi,superprocess}
\selectlanguage{english}

\end{document}